\documentclass[a4paper,11pt]{article}
\usepackage[utf8]{inputenc}
\usepackage[english]{babel}

\usepackage{amsmath}
\usepackage{amssymb}
\usepackage{amsthm}
\usepackage{enumerate}
\usepackage{fullpage}
\usepackage{mathpazo} 
\usepackage{natbib}
\usepackage{tikz}
\usetikzlibrary{calc}
\usepackage{subcaption}
\usepackage[normalem]{ulem}
\usepackage{hyperref}
\usepackage{cleveref}
\usepackage[nocomma]{optidef}
\usepackage{comment}
\usetikzlibrary{math}
\usepackage{longtable,tabu}

\newtheorem{theorem}{Theorem}[section]

\newtheorem{lemma}[theorem]{Lemma}
\newtheorem{observation}[theorem]{Observation}
\newtheorem{proposition}[theorem]{Proposition}
\newtheorem{corollary}[theorem]{Corollary}
\theoremstyle{definition}

\newtheorem{definition}[theorem]{Definition}

\newcommand{\ZnZ}{\mathbb{Z}/n\mathbb{Z}}
\newcommand{\dR}{\mathbb{R}}
\newcommand{\dN}{\mathbb{N}}

\newcommand{\OptVal}{\operatorname{OptVal}}
\pgfdeclarelayer{bg} 
\pgfsetlayers{bg,main}

\newcommand\DoubleLine[7][3pt]{%
	\path(#2)--(#3)coordinate[at start](h1)coordinate[at end](h2);
	\draw[#4]($(h1)!#1!90:(h2)$)-- node [auto=left] {#5} ($(h2)!#1!-90:(h1)$); 
	\draw[#6]($(h1)!#1!-90:(h2)$)-- node [auto=right] {#7} ($(h2)!#1!90:(h1)$);
}

\definecolor{darkred}{rgb}{0.70,0.00,0.12}
\definecolor{lightgrey}{rgb}{0.8,0.8,0.8}
\definecolor{lightblue}{rgb}{0.8,0.8,1}
\definecolor{darkblue}{rgb}{0.00,0.12,0.70}
\definecolor{darkgreen}{rgb}{0.00,0.30,0.05}
\definecolor{lightgreen}{rgb}{0.23,0.7,0.0}
\definecolor{lightpurple}{rgb}{1.0,0.1,1.0}
\definecolor{forestgreen}{rgb}{0.08,0.50,0.00}
\definecolor{limegreen}{rgb}{0.71,0.90,0.11}

\title{The Price of Symmetric Line Plans in the Parametric City}
\author{%
    \begin{tabular}{ccc}
        Berenike Masing\footnote{Zuse Institute Berlin, Takustr.\ 7, 14195 Berlin, Germany} \footnote{Funded by the Deutsche Forschungsgemeinschaft (DFG, German Research Foundation) under Germany's Excellence Strategy – The Berlin Mathematics Research Center MATH+ (EXC-2046/1, project ID: 390685689).} &
        Niels Lindner\footnotemark[1] &
        Ralf Borndörfer\footnotemark[1] \\
        \texttt{masing@zib.de} &
        \texttt{lindner@zib.de} &
        \texttt{borndoerfer@zib.de} 
    \end{tabular}
}

\date{\today}

\begin{document}

\maketitle

\begin{abstract}
    \noindent We consider the line planning problem in public transport in the Parametric City, an idealized model 
    that captures 
    typical scenarios by a (small) number of parameters. The Parametric City is rotation symmetric, but 
    optimal line plans are not always symmetric. This raises the question to quantify the symmetry gap between the best symmetric and the overall best solution. For our analysis, we formulate the line planning problem as a mixed integer linear program, that can be solved in polynomial time if the solutions are forced to be symmetric. The symmetry gap is provably small when a specific Parametric City parameter is fixed, and we give an approximation algorithm for line planning in the Parametric City in this case. While the symmetry gap can be arbitrarily large in general, we show that symmetric line plans are a good choice in most practical situations.
\end{abstract}

\section{Introduction}

The overall goal of public transport planning is to satisfy the demand providing a user-friendly service, while limiting the operator's cost,
see \citet{CEDER1986}, \citet{Assad}, and \citet{Bussieck}. 
An overview tailored to line planning in particular is given by \citet{Schoebel:overview}. \citet{Karbstein:Thesis} as well as \citet{Schmidt} consider passenger-oriented models, which focus on reducing transfers. 
The basic transit network design problem, which can include line planning, is surveyed by \citet{KepaptsoglouKarlaftis} and  \citet{Lopez}. Most recent approaches focus on representing a specific city as realistically as possible by tailor made ``irregular networks''. This produces the best possible results for the moment of consideration, but it makes generalization, benchmarking, and extrapolation hard, as network structures do not easily carry over to other cities, and not even to the same city twenty years in the future. An alternative approach is to consider \emph{generic city structures}, that reflect the most prevalent economical and spatial aspects, but remain simple enough to be analyzed and understood. Well organized prototypcial (or even standardized) transportation networks could then be used in similar cities. Two examples are Manhattan type grids that were considered by  \citet{Holroyd}, 
and ring-radial models as studied by \citet{BYRNE197597} and more recently by \citet{Badia}. There have also been hybrid approaches, e.g., \citet{Daganzo} combines a grid-like model at the center with a hub-and-spoke structure surrounding it.
A versatile model, and the one that we consider in this paper, is the so-called Parametric City by \citet{Fielbaum:ModelIntro}. It provides an idealized representation of a city that can be adapted to different situations by a flexible choice of parameters.


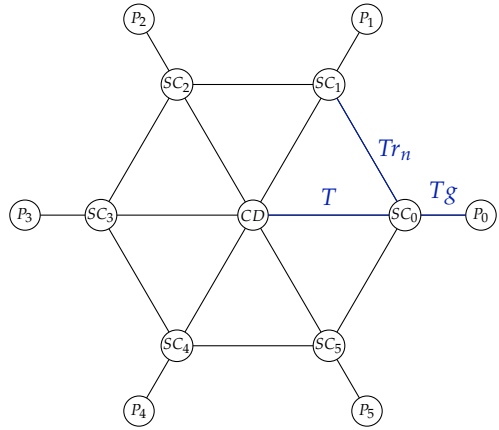
\begin{figure}[ht]
    \centering
\begin{tikzpicture}[myn/.style={draw,circle,fill=none, font=\footnotesize, outer sep=0pt, inner sep=0pt,minimum size=0.4cm}]
\tikzmath{\n = 6; \nm1 = \n-1; \nm2 = \n-2; \T = 2; \gT = \T*(1+0.5);}
			\node[myn] (CD) at (0,0) {\tiny{$CD$}};
			\foreach \x in {0,...,\nm1}{
				\node[myn] (SC\x) at ({\x*360/\n}: \T cm) {\tiny{$SC_{\x}$}};
			}
			
			\foreach \x in {0,...,\nm1}{
				\node[myn] (P\x) at ({\x*360/\n}: \gT cm) {\tiny{$P_{\x}$}};
			}
			\foreach \x in {0,...,\nm1}{%
			
					\draw(SC\x) -- (CD);
					\draw(SC\x) -- (P\x);
				
				}
				
			\foreach \x in {0,...,\nm1}{%
			    \pgfmathtruncatemacro {\y}{mod(round (1+\x),\n)}
			    \draw(SC\x) -- (SC\y);
				}
				
		\draw[darkblue] (CD) -- node[midway, above, font = \footnotesize] {$T$} (SC0);
		\draw[darkblue] (SC0) -- node[midway, right, font = \footnotesize] {$Tr_n$} (SC1);
		\draw[darkblue] (SC0) -- node[midway, above, font = \footnotesize] {$Tg$} (P0);

\end{tikzpicture}
\caption{The infrastructure graph of the Parametric City for $n = 6$ }
\label{fig:exampleParametricCity}
\end{figure}

\begin{figure}[ht]
\begin{subfigure}{.5\textwidth} 
\begin{tikzpicture}[myn/.style={draw,circle,fill=none, font=\footnotesize, outer sep=2pt, inner sep=0pt,minimum size=0.4cm},
dot/.style={circle, draw, fill=black, inner sep=0pt, minimum width=3pt, outer sep=5pt},]
\tikzmath{\n = 8; \nm1 = \n-1; \nm2 = \n-2; \T = 2; \gT = \T*(1+1/2);}
			\node[myn]
			(CD) at (0,0) {\tiny{$CD$}};
			\foreach \x in {0,...,\nm1}{
				\node[myn]
				(SC\x) at ({\x*360/\n}: \T cm) {\tiny{$SC_{\x}$}};
			}
			
			\foreach \x in {0,...,\nm1}{%
				\node[myn]
				(P\x) at ({\x*360/\n}: \gT cm) {\tiny{$P_{\x}$}};
			}
			\foreach \x in {0,...,\nm1}{%
					
				\DoubleLine{SC\x}{P\x}{<-,blue}{}{-> ,blue}{};
				
				}
				
			\foreach \x in {1,...,\nm1}{%
			    \pgfmathtruncatemacro {\y}{mod(round (1+\x),\n)}
			    	\draw[cyan,->] (SC\x) -- (SC\y);
			    	}

			\draw[cyan,->] (SC0) -- (CD); 
			\draw[cyan,->] (CD) to (SC1);
        \node[align=center, font=\footnotesize,rectangle,draw] (legend) at (0 , -\gT -1) {
        Frequencies: \\$\textcolor{blue}{\rightarrow,} \textcolor{cyan}{\rightarrow} 1$  };
\end{tikzpicture}
    \label{fig:ExtremeAsymmetricLinePlan}
\end{subfigure}%
\begin{subfigure}{.5\textwidth} 
\begin{tikzpicture}[myn/.style={draw,circle,fill=none, font=\footnotesize, outer sep=2pt, inner sep=0pt, minimum size=0.4cm},
dot/.style={circle, draw, fill=black, inner sep=0pt, minimum width=3pt, outer sep=5pt},]
\tikzmath{\n = 8; \nm1 = \n-1; \nm2 = \n-2; \T = 2; \gT = \T*(1+1/2);}
			\node[myn]
			(CD) at (0,0) {\tiny{$CD$}};
			\foreach \x in {0,...,\nm1}{
				\node[myn]
				(SC\x) at ({\x*360/\n}: \T cm) {\tiny{$SC_{\x}$}};
			}
			
			\foreach \x in {0,...,\nm1}{%
				\node[myn]
				(P\x) at ({\x*360/\n}: \gT cm) {\tiny{$P_{\x}$}};
			}
			\foreach \x in {0,...,\nm1}{%
					
				\DoubleLine{SC\x}{P\x}{<-,blue}{}{-> ,blue}{};
				
				}
				
			\foreach \x in {0,...,\nm1}{%
			    \pgfmathtruncatemacro {\y}{mod(round (1+\x),\n)}
			    \DoubleLine{SC\x}{SC\y}{<-,lightpurple}{}{->,dashed ,lightpurple}{};
				}
				
			\foreach \x in {2,5}{
			\DoubleLine{SC\x}{CD}{<-,darkgreen}{}{->, darkgreen}{};
				};
			\foreach \x in {4,7}{
			\DoubleLine{SC\x}{CD}{<-,darkgreen, dashed}{}{->, darkgreen, dashed}{};
				};
			\foreach \x in {1,6}{
			\DoubleLine{SC\x}{CD}{<-,orange}{}{->, orange}{};
				};
			\foreach \x in {0,3}{
			\DoubleLine{SC\x}{CD}{<-, dashed,orange}{}{->, dashed,orange}{};
				};
				
			\draw[cyan,->] (SC1) -- (SC0);
			\draw[cyan,->] (SC0) -- (CD); 
			\draw[cyan,->] (CD) to (SC1);
        \node[align=center, font=\footnotesize,rectangle,draw] (legend) at (0 , -\gT -1) {
        Frequencies: \\$\textcolor{blue}{\rightarrow} 24$ \quad $\textcolor{darkgreen}{\rightarrow, \dashrightarrow} 7$  \quad $\textcolor{orange}{\rightarrow,\dashrightarrow} 6$ \quad $\textcolor{lightpurple}{\rightarrow, \dashrightarrow,} \textcolor{cyan}{\rightarrow} 1$  };

\end{tikzpicture}
\end{subfigure}
\caption{Examplary Asymmetric Optimal Line Plans}
\label{fig:AsymmetricLinePlans}
\end{figure}
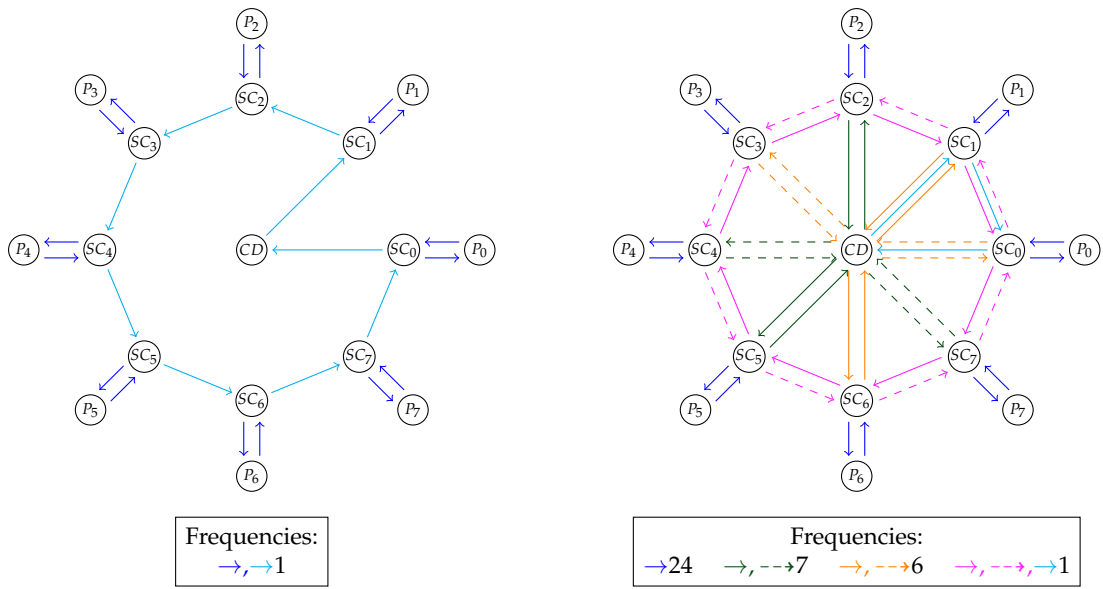

The main degree of freedom in the Parametric City is the design of the line system, i.e., the selection of the routes and their frequencies of operation, and one usually aims at combining the minimization of operator costs and the maximization of user comfort (e.g., short travel times, few changeovers, low waiting times). As the Parametric City is rotation symmetric in both its geometry and demand pattern, see an exemplary graph in Figure~\ref{fig:exampleParametricCity}, one would expect that an optimal public transport system is equally symmetric, and, in particular, that the line system is symmetric. This is desirable because a symmetric system has a simple structure that makes it easy to operate and memorize. While, as it will turn out, this is true in most cases, there are instances in which it is possible to find asymmetric line plans that have better objective values than any symmetric line plan. Indeed, \Cref{fig:AsymmetricLinePlans} presents two exemplary optimal line plans for two instances of the Parametric City. Each line is depicted as a sequence of colored arrows, indicating its direction and a corresponding frequency. The left instance is highly artificial: Here we set the capacity of the vehicles to the total patronage, which means that all passengers fit into a single vehicle. Thus, we just need to ensure that there is a line visiting all stations in order to pick up all passengers, which results in the illustrated -- clearly asymmetric -- optimal line plan. However, there are also examples with more realistic parameter choices which also result in asymmetric line plans, such as the one depicted on the right: Here, there are two orange lines of frequency $6,$ two green lines of frequency $7$ and even a circular line traveling clockwise between $SC_1, SC_0$ and $CD$ with frequency $1.$ Instead of, for example, using only central lines of frequency $7,$ it is beneficial to force a few of the passengers from $SC_3$ with destination $CD$ to take a detour via the neighboring subcenters $SC_2$ or $SC_4.$ By that, we can set the frequency to $6$ on the orange line (which then runs at full capacity), while the green lines with frequency $7$ can still fit the surplus of passengers from their neighbors. 
As neither the passenger paths, nor the line plans are rotated copies of each other, this line plan is asymmetric. 

We investigate in this paper the properties of symmetric line plans in comparison to asymmetric ones within the Parametric City. When do symmetric input data result in optimal symmetric line plans? Under which circumstances are symmetric line plans a particularly bad choice? Our aim is to provide planners with a guideline: When are they justified in assuming symmetry? When can they use a symmetric solution as a reasonable approximation? In which cases is it detrimental to assume symmetry? 
To answer these questions, we model the line planning problem in the Parametric City as a mixed integer program, taking all possible lines into account (including unidirectional ring lines). It turns out that the properties of the Parametric City allow to reformulate this model in an arc-based way, which is computationally efficient.
With this model, we can quantify the \emph{Symmetry Gap} between optimal symmetric and asymmetric line plans. It turns out to be small in numerical experiments, and it is also possible to derive analytic bounds. In particular, if one of the parameters of the Parametric City is fixed, namely, the distance factor between the subcenters and their peripheries $g$, a $\left(1+\frac{1+\sqrt{2}}{g} \right)$-factor approximation algorithm (with respect to the symmetry gap) can be derived. These results justify the use of symmetric line plans in practice. If the parameter $g$ is not restricted, instances with arbitrary symmetry gaps exist.

\section{Line Planning in the Parametric City}
\label{sec:parametricCity}

The Parametric City is described by three topological parameters, the number $n\in\dN$ of radial connections, a radius $T>0$, and an offset factor $g>0$, and five  demand parameters, the patronage $Y>0$, and shares $0<a,\alpha,\gamma,\tilde{\alpha},\tilde{\gamma}<1$. Line planning involves three further parameters, an objective weighing factor $\mu \in [0,1]$, a vehicle capacity $K>0$, and a frequency bound $\Lambda>0$.

\subsection{The Parametric City}


The Parametric City \citep{Fielbaum:ModelIntro} consists of an infrastructure graph and an associated origin-destination matrix. The graph is a planar embedding of a \emph{helm graph} \citep{helm} $\mathcal G = (V, A)$\label{ass:G}\label{ass:V}\label{ass:A} with $2n+1$ vertices and $3n$ pairs of anti-parallel arcs, where $n \geq 4$\label{ass:n} is some natural number (see \Cref{fig:exampleParametricCity} for $n = 6$).
The central node $CD$\label{ass:CD} at the origin is the \emph{central business district}. It is connected to a cycle of $n$ surrounding nodes $SC_0, SC_1, \dots SC_{n-1}$\label{ass:S}, called \emph{subcenters}, each of which is connected to a \emph{periphery} node $P_0, P_1, \dots P_{n-1}$\label{ass:Pi}. The subcenters and peripheries are placed around $CD$ at distances $T >0$\label{ass:T} and $T(1+g),g>0,$\label{ass:g} at coordinates
\begin{align*}
SC_j &= T [\sin(j \, 2\pi/n), \cos(j \, 2\pi/n)]^\top \quad &\text{ for } j \in \ZnZ, \\
P_j &= T(1+g) [\sin(j \, 2\pi/n), \cos(j \, 2\pi/n)]^\top \quad &\text{ for } j\in \ZnZ.
\end{align*}
The choice of the parameters $n, T$, and $g$ completely determines the size and shape of the graph underlying the Parametric City. In particular, periphery $P_i$ is placed on the continuation of the central axis through $CD$ and $SC_i$, and the distance between two adjacent subcenters is $ r_n T = 2 \sin(\pi/n) T$\label{ass:r}. In general, we denote by $\tau_{a}$\label{ass:tau} the length of an arc $a\in A.$  


\begin{table}[ht]
	\centering
	\begin{tabular}{|l||c | c | c |c |c |}
		\hline
	$s,t$&$P_i$ & $P_{j}, {j\neq i}$ &$SC_i$ & $SC_{j},  {j\neq i}$ &$CD$\\

		\hline 
		$P_i$ & $0$ & $0$ & $\frac{aY}{n} \beta $ & $\frac{aY}{n (n-1)} \gamma $ & $\frac{aY}{n} \alpha$\\   
		$SC_i$ & 0 & 0 & 0 & $\frac{(1-a)Y}{n(n-1)}\tilde{\gamma}$ & $\frac{(1-a)Y}{n}\tilde{\alpha}$\\
		$CD$ & 0 & 0 & 0 & 0 & 0\\
		\hline
	\end{tabular}
	\caption{Demand $d_{s,t}$ of the Parametric City}
	\label{tab:OD}
\end{table}

The demand pattern is described by the origin-destination matrix $OD = \left(d_{s,t}\right)_{(s,t) \in V\times V}$ according to \Cref{tab:OD} for $i,j \in \ZnZ.$\label{ass:d} It is supposed to model the morning rush hour: The peripheries are  considered to be pure trip generators, $CD$ is a pure trip attractor, and the subcenters are mixed districts, both generating and attracting trips. Parameter $Y>0$\label{ass:Y} is the patronage. Parameter $a \in ]0,1[$\label{ass:a} controls the share of travelers originating from the peripheries, so $a\, Y$ passengers come from a periphery, while $(1-a)Y $ start in a subcenter and none start at $CD.$ 
Parameters $\alpha, \beta, \gamma \in ]0,1[$\label{ass:alphaBetaGamma} are the percentages of trips from the peripheries to  $CD,$ their own subcenter, and the other subcenters, respectively. Similarly, $\tilde{\alpha} \in ]0,1[$\label{ass:tildeAlpha} are the percentages of trips from the subcenters to $CD$ and $\tilde{\gamma}\in ]0,1[$\label{ass:tildeGamma} the percentage of trips from the subcenters to other subcenters. It holds $\frac{\alpha}{\tilde{\alpha}} = \frac{\gamma}{\tilde{\gamma}}$ or equivalently $\tilde{\alpha}  = \frac{\alpha}{\alpha+\gamma}, \tilde{\gamma} , = \frac{\gamma}{\alpha + \gamma}$ since $\alpha + \beta + \gamma = 1$ and $\tilde{\alpha}+\tilde{\gamma} = 1$ because it is assumed that $CD$ is as attractive from a periphery as from a subcenter. 
Note that -- while the $OD$-matrix is not symmetric in the classical sense, i.e., we have $d_{P_0, CD} \neq d_{CD,P_0}$ -- the demand itself is rotation symmetric: For example, the demand from a periphery to the central business district is the same as that of any other periphery to $CD.$ 

A city like Berlin, which has no dominating center, but many (roughly) equally important subcenters, can be represented by choosing relatively large values for $n, T, g$ to mimic the geometry, and a fairly large value for $\gamma$ to model the dispersion of trips within the city.

\subsection{The Line Planning Problem}
\label{sec:linePlanning}
The line planning problem $(LPP)$ in the Parametric City can be formulated as a mixed integer program using two types of variables: $y_p \in \dR$\label{ass:y} for the passenger flow on path $p \in P$, and $f_l \in \dN$\label{ass:f} for the frequency of line $l\in L$, see, e.g., \citet{Borndorfer:Column}. Here, $P$\label{ass:P} is the set of all simple paths, while $L$\label{ass:L} is the line pool consisting of all simple directed cycles in $\mathcal{G}$ (a typical line will either be an unidirectional ``ring line'' or a path that is traversed forth and back). We refer to $f =(f_l)_{l \in L}$ as the line plan and say that a line $l$ is part of the line plan if and only if $f_l> 0$; the same holds for the passenger flow. We define the sets $P_a$\label{ass:Pa} and $L_a$~\label{ass:La} as the set of paths and lines which use arc $a\in A$, respectively. Further, $P_{s\to t}$\label{ass:Pst} is the set of all $s$-$t$-paths.
The resulting model is as follows.

\begin{definition}[Line Planning Problem]
	\label{MILP} 
	\begin{mini!}|s|[2]
		{f,y}{\mu \sum_{l\in L} \tau_l \textcolor{black}{f_l} + (1-\mu) \sum_{p \in P} \tau_p \textcolor{black}{y_p} =: cost(f,y)}
		{}
		{\textcolor{black} {(LPP)(K,\Lambda,\mu)\quad} \nonumber} 
		\addConstraint{\sum_{p \in P_{s \to t}} \textcolor{black}{y_p}}{= d_{st} \quad \label{con:demand}} {\forall (s,t) \in D}
		\addConstraint{\sum_{p \in P_a} \textcolor{black}{y_p} - \sum_{l \in L_a} \textcolor{black}{f_l} K}{\leq 0 \label{con:capacity} \quad}{\forall a \in A}
		\addConstraint{\sum_{l \in L_a} \textcolor{black}{f_l}} {\leq \Lambda  \label{con:arcCapacity}}{\forall a \in A}
		\addConstraint{\textcolor{black}{f_l}}{\in \mathbb{N}}{\forall l \in L \label{con:frequencies}}
		\addConstraint{\textcolor{black}{y_p} }{\geq 0 \label{con:paths}}{\forall p \in P}
	\end{mini!}
\end{definition}
The \emph{passenger flow conditions} \eqref{con:demand} in combination with the non-negativity constraints \eqref{con:paths} ensure that passengers get routed and demand is met. The vehicle capacity is denoted by the constant $K>0.$\label{ass:K} 
The \emph{capacity constraints} \eqref{con:capacity} guarantee that there is sufficient vehicle space to transport all passengers on arc $a\in A.$ Finally, the \emph{street capacity constraints} \eqref{con:arcCapacity} prevent overcrowded streets by ensuring that the accumulated frequency of an arc is not larger than a parameter $\Lambda > 0.$\label{ass:Lambda}
The objective is a combination of operator and user costs that are weighed by parameter $\mu \in [0,1]$~\label{ass:mu}; this is standard in most line planning approaches.
The running and travel times are equated with the total length of a line or path, i.e., $\tau_l = \sum_{a \in l} \tau_a$~\label{ass:tauL} and $\tau_p = \sum_{a \in p} \tau_a,$~\label{ass:tauP}\label{ass:in} where $\tau_a$ is the length of arc $a \in A$ with respect to the Parametric City model, i.e.,
\begin{equation} \label{ass:tau2}
    \tau_{a} =\begin{cases}
    T \quad & a \in \{(CD, SC_i), (SC_i,CD)\; \mid \; i \in \ZnZ\},\\
    gT \quad & a \in \{(P_i, SC_i), (SC_i, P_i) \; \mid  \; i \in \ZnZ\}, \\
    r_n T \quad & a \in \{(SC_i, SC_{i \pm 1}) \; \mid \; i \in \ZnZ \}; 
    \end{cases}
\end{equation}
here and elsewhere, we write $i\pm q$ for $i, q \in \ZnZ$ instead of $(i \pm q)  \bmod n$.



\subsection{An Arc-Based Model}
The main computational difficulty with the line planning model $LPP$ is the large number of integer line frequency variables, which is exponential if the line pool $L$ is completely unrestricted (and still quadratic in the number of nodes if $L$ is restricted to simple cycles). However, the line variables are always aggregated over arcs $F_a := \sum_{l\in L_a} f_l$; in particular, for the objective holds
\begin{equation}
\sum_{l\in L} \tau_l f_l = \sum_{l \in L} \sum_{a \in l} \tau_a f_l = \sum_{a \in A} \sum_{l \in L_a} \tau_a f_l = \sum_{a \in A} \tau_a F_a. 
\label{eq:lineToArc}
\end{equation}
The aggregation $F$\label{ass:F} induces a circulation, and conversely, any integer circulation can be decomposed into an equivalent set of lines. 
The resulting arc-based line planning model $ALPP$ has only $6n$ integer variables and reads as follows:

\begin{definition}[Arc-Based Line Planning Problem]
 \label{MILP:A} 
\begin{mini!}|s|[2]<b>
	{F,y}{\mu \sum_{a\in A} \tau_a \textcolor{black}{F_a} + (1-\mu) \sum_{p \in P} \tau_p \textcolor{black}{y_p} =: cost_A(F,y)}
	{}
	{\textcolor{black} {(ALPP)(K,\Lambda,\mu)\quad} \nonumber} 
	\addConstraint{\sum_{p \in P_{s \to t}} \textcolor{black}{y_p}}{= d_{s,t} \quad \label{con:demand_A}} {\forall (s,t) \in V\times V}
	\addConstraint{\sum_{p \in P_a} \textcolor{black}{y_p} - \textcolor{black}{F_a} K}{\leq 0 \label{con:capacity_A} \quad}{\forall a \in A}
	\addConstraint{\textcolor{black}{F_a}}{\leq \Lambda \label{con:arcCapacity_A}}{\forall a \in A}
	\addConstraint{\sum_{a \in \delta^+(v)} \textcolor{black}{F_{a}} - \sum_{a \in \delta^-(v)}  \textcolor{black}{F_{a} }} {= 0\quad \label{con:flow_A}}    {\forall v \in V} 
	\addConstraint{\textcolor{black}{F_a}}{\in \mathbb{N}}{\forall a \in A \label{con:frequencies_A}}
	\addConstraint{\textcolor{black}{y_p} }{\geq 0 \label{con:paths_A}}{\forall p \in P}
\end{mini!}
\end{definition}
\label{ass:delta}

The arc-based formulation is not only smaller, it will also turn out that it is easier to analyze than the line-based one. We note the following properties of feasible solutions of $ALPP:$

\begin{observation}
\label{lem:properties}
If $(F,y)$ is feasible for $ALPP$, then 
\begin{enumerate}
    \item $F_{P_i, SC_i} = F_{SC_i, P_i} \geq \left \lceil \frac{Y a }{ n K} \right \rceil$ for all $i \in \ZnZ,$ with equality if $(F,y)$ is optimal and $\mu \neq 0$. 
    \label{property:Fps=Fsp}
    \item $F_{(SC_k, CD)} > 0$ for some $k \in \ZnZ$. \label{prop:sc}
\item $F_{(CD,SC_k)} > 0$ for some $k \in \ZnZ$.
\end{enumerate}
\end{observation}


\section{Symmetry}
\label{sec:symmetry}

While the graph $G = (V, A)$ of the Parametric City is clearly rotation symmetric, it is not so clear what a symmetric demand, passenger flow, or line plan is, and in fact, different notions can be considered. 

\begin{definition}[Rotation]
    We identify a vertex $v \in V$ with its coordinates in the plane and define its rotation $\rho_z$ by the angle $ z \cdot 2 \pi/n $ around the origin, $z\in\ZnZ$, as 
	\begin{align*}
	\rho_z: \ &V \rightarrow V, \quad
	 v = \begin{bsmallmatrix} v_1 \\ v_2 \end{bsmallmatrix} \mapsto  R_z v = 
	 \begin{bsmallmatrix} cos(\frac{2\pi} {n} \, z ) & -sin(\frac{2\pi} {n} \, z )\\
	sin(\frac{2\pi} {n} \, z ) & \hphantom{-}cos(\frac{2\pi} {n} \, z ) 
	\end{bsmallmatrix}\begin{bsmallmatrix} v_1 \\v_2 \end{bsmallmatrix}
	\end{align*}
	Rotations can be extended to arcs and arbitrary vertex tuples componentwise:\label{ass:rotation}
	\begin{align*}
	\rho_z:\ & V^m \rightarrow V^m, 
	\quad (v_0,v_1,\dots, v_{m-1})  \mapsto  
	 (R_z v_0, R_z v_1, \dots, R_z v_{m-1}) .
	\end{align*}
	\end{definition}

Rotating the vertices of the Parametric City by $\rho_z$ results in
	\begin{alignat*}{3}
	&\rho_z( CD)\: &=& R_z \begin{bsmallmatrix} 0 \\0 \end{bsmallmatrix} = CD,\\
	&\rho_z( SC_j)\: &=& R_z SC_j = T \begin{bsmallmatrix} cos(\frac{(j+z) 2\pi}{n})  \\ sin(\frac{(j+z) 2\pi}{n}) \end{bsmallmatrix} = SC_{j+z},\\
	&\rho_z(P_j) &=& R_z P_j = T(1+g) \begin{bsmallmatrix} cos(\frac{(j+z) 2\pi}{n})  \\ sin(\frac{(j+z) 2\pi}{n}) \end{bsmallmatrix} = P_{j+z}.
\end{alignat*} Consequently,  subcenters get rotated onto subcenters, peripheries onto peripheries, and the central business district remains fixed.
We regard a property of the Parametric City to be symmetric if it has the same value for all rotations $\rho_z$.  
The demand is then rotation symmetric in this sense.
In the same manner, we define symmetric solutions:

\begin{definition}[Symmetric Solution]
	Consider a solution $(f,y)$ to $LPP$ and the equivalent solution $(F,y)$ to $ALPP.$ 
	\begin{enumerate}
		\item  The line plan $f$ is \emph{line-symmetric} if for all $l \in L:$
		\; $f_l = f_{\rho_z(l) }$ for all   $z \in \ZnZ.$ 
		\label{it:lineplan}
		\item The passenger flow $y$ is \emph{path-symmetric} if for all $p \in P:$
		\; $y_p = y_{\rho_z(p)}$ for all $z \in \ZnZ.$ 
		 \label{it:pathplan}
		\item The line plan $f$ is \emph{arc-symmetric} if for all $a \in A:$
		\; $\sum_{l \in L_a} f_l = \sum_{l \in L_{\rho_z ( a) }} f_l$ for all $ z \in \ZnZ.$ 
		\item The frequency plan $F$ is \emph{arc-symmetric} if 
		$F_{a} = F_{\rho_z ( a )} $ for all $ z \in \ZnZ.$ \label{it:arcsymmetry} 
	\end{enumerate}
	 The solution $(f,y)$ is called \emph{symmetric} if Conditions~\ref{it:lineplan} and~\ref{it:pathplan} hold, while $(F,y)$ is \emph{symmetric} if Conditions~\ref{it:pathplan} and~\ref{it:arcsymmetry} hold.
\end{definition}


These definitions give rise to the following sequence of results. 

\begin{lemma}
Any symmetric line plan $f$ is is arc-symmetric.
\end{lemma}

\begin{lemma}
\label{lem:fourvalues}
An arc-symmetric frequency plan of a feasible solution $(F,y)$ has the following properties: 
\begin{itemize}
    \item $F_{(P_0, SC_0)}= F_{(P_z, SC_z)} = F_{(SC_z, P_z)} \in \mathbb{N}_{{>0}} $ for all $z \in \ZnZ$,
    \item $F_{(SC_0, SC_1)} = F_{(SC_z, SC_{z+1})}$ for all $z \in \ZnZ$,
    \item $F_{(SC_0, SC_{n-1})} = F_{(SC_z, SC_{z-1})}$ for all $z \in \ZnZ$,
    \item $F_{(CD, SC_0)}= F_{(CD, SC_z)} = F_{(SC_z, CD)} \in \mathbb{N}_{{>0}} $ for all $z \in \ZnZ$.
\end{itemize}
We denote the corresponding frequency values by $  \ F_P,\  F_{S+}, \ F_{S-}$, and $F_{C},$ i.e.,  ${F_P := F_{(P_0, SC_0)},}$ \linebreak ${F_{S+} := F_{(SC_0, SC_1)},} $ $F_{S-} := F_{(SC_1, SC_0)}$ and $F_{C} := F_{(SC_0, CD)}$~\label{ass:FCSP},  see Figure \ref{fig:symmetricLP}.
\end{lemma}
			


\begin{lemma}
\label{lem:symmF_implies_symmf}
For any arc-symmetric solution $(F,y)$ to $ALPP,$ there exists a line-symmetric solution $(f,y)$ to $LPP$ with the same objective value.
\end{lemma}
\begin{proof}
As the previous lemma states, an arc-symmetric solution has at most four different frequencies ($F_P$ on the peripheral arcs, $F_{C}$ on the axes incident to $CD$, and $F_{S-}$ and $F_{S+}$ on the arcs between subcenters), see again Figure \ref{fig:symmetricLP}. These frequencies induce a canonical symmetric line plan. 
\end{proof}

\begin{figure}[!ht]
\centering
	\begin{tikzpicture}[myn/.style={draw,circle,fill=none, font=\footnotesize, outer sep=2pt, inner sep=0pt, minimum size=0.4cm},
		dot/.style={circle, draw, fill=black, inner sep=0pt, minimum width=3pt, outer sep=5pt}, scale= 0.8]
		\tikzmath{\n = 8; \nm1 = \n-1; \nm2 = \n-2; \T = 2; \gT = \T*(1+1/2);}
		\node[myn]
		(CD) at (0,0) {\tiny{$CD$}};
		\foreach \x in {0,...,\nm1}{
			\node[myn]
			(SC\x) at ({\x*360/\n}: \T cm) {\tiny{$SC_{\x}$}};
		}
		
		\foreach \x in {0,...,\nm1}{%
			\node[myn]
			(P\x) at ({\x*360/\n}: \gT cm) {\tiny{$P_{\x}$}};
		}
		\foreach \x in {0,...,\nm1}{%
			
			\DoubleLine{SC\x}{P\x}{<-,blue}{}{-> ,blue}{};
			
		}
		
		\foreach \x in {0,...,\nm1}{%
			\pgfmathtruncatemacro {\y}{mod(round (1+\x),\n)}
			\DoubleLine{SC\x}{SC\y}{<-,lightpurple}{}{->, orange}{};
		}
		
		\foreach \x in {0,...,\nm1}{
			\DoubleLine{SC\x}{CD}{<-,darkgreen}{}{->, darkgreen}{};
		};
		
		\node[align=center, font=\footnotesize,rectangle,draw] (legend) at (0 , -\gT -1) {
        Frequencies: \\$\textcolor{blue}{\rightarrow} F_P$ \quad $\textcolor{darkgreen}{\rightarrow} F_C$  \quad $\textcolor{orange}{\rightarrow} F_{S+}$ \quad $\textcolor{lightpurple}{\rightarrow} F_{S-} $};
	\end{tikzpicture}
    \caption{Symmetric Frequency/Line Plan}
    \label{fig:symmetricLP}
	\end{figure}
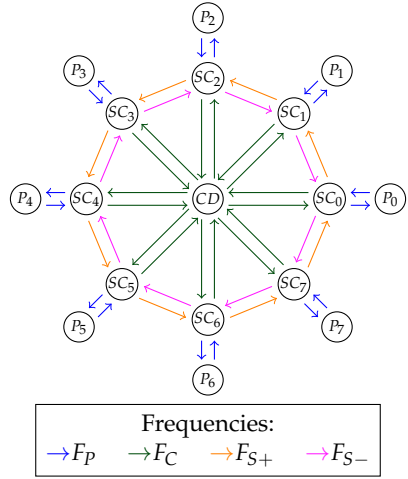	

\begin{lemma}[Construction of a Symmetric Solution]
\label{lem:ConstructionSymmetricSolution}
For any feasible solution $(F, y)$ of $ALPP,$ a symmetric solution $(F^s, y^s)$ can be constructed as follows: 
$$F^s_a:=\left \lceil \frac{1}{n}  \sum_{z = 0}^{n-1} F_{\rho_z(a)} \right \rceil \quad \text{for } a \in A, \quad \quad y^s_p := \frac{1}{n} \sum_{z = 0}^{n-1} y_{\rho_z(p)} \quad \text{for }  p \in P.$$
The user costs $\sum_{p \in P} \tau_p y_p$ remain constant under this symmetrization.
\end{lemma}
\begin{proof}
 For a feasible solution $(F,y)$ of $ALPP(K,\Lambda,\mu)$, let $\rho_z(F)$ and $ \rho_z(y)$ denote the rotated frequency plan and passenger flow, respectively. Clearly, if $(F,y)$ is feasible, then any rotation $(\rho_z(F), \rho_z(y)$ is feasible as well. Consequently $\sum_{z = 0}^{n-1} (\rho_z(F),\rho_z(y))/n$ is feasible for the LP relaxation of $ALPP(K,\Lambda,\mu)$. 
 
 
 Rounding up $\sum_{z = 0}^{n-1} \rho_z(F)/n$ to $F^s$ is no problem in the capacity constraints $\sum_{p\in P_a} y_p - F_a K \leq 0$ (\ref{con:capacity_A}). This is also true for the street capacity constraints $F_a \leq \Lambda$ (\ref{con:arcCapacity_A}). Indeed, for any arc $a$, they hold in particular for $F_{a_{max}} := \max\{ F_{a'} \mid a' = \rho_z(a), z \in \ZnZ\}\geq \frac{1}{n} \sum_{z=0}^{n-1} F_{\rho_z(a)}$, and as $F_{a_{max}}$ is integer, they hold for $F^s_a$. The flow conservation constraints (\ref{con:flow_A}) are also fulfilled: At node $CD$ we have $$\sum_{a \in \delta^+(CD) } F_a = \sum_{a \in \delta^-(CD)} F_a.$$
	This implies that $F^s_a = F^s_{\bar{a}}$ for all $a \in \cup_{z =0}^{n-1} (SC_z, CD), \bar{a} \in \cup_{z=0}^{n-1} (CD, SC_z). $
	The same holds at each node $P_j, z \in \ZnZ.$ Because $F^s$ is symmetric, we have $F^s_{SC_0, SC_{1}} = F^s_{\rho_z(SC_0, SC_1)},$ in particular $F^s_{SC_j, SC_{j+1}} = F^s_{SC_{j-1}, SC_j};$ as well as $F^s_{SC_{j+1}, SC_{j}} = F^s_{SC_{j}, SC_{j-1}}.$ Consequently, flow is preserved in each node $SC_j$ as well. Finally, the user costs remain the same, since they depend only
on path lengths, which are identical for every rotation.
\end{proof}

\begin{corollary}
\label{cor:infeasibilityCond}
$ALPP$ is feasible if and only if a symmetric solution exists.
\end{corollary}

\begin{lemma}
	\label{lem:symFtoSymY}
	If $(F,y)$ is a feasible arc-symmetric solution to $ALPP,$ there exists a symmetric passenger flow $\tilde{y}$ such that the solution $(F, \tilde{y})$ is feasible and $cost_A(F,y) = cost_A(F,\tilde{y}).$
	
	The same holds for a feasible symmetric solution $(f,y)$ of $LPP.$
	\begin{proof}
		Construct $(F,\tilde{y}):=(F^s,y^s)$ according to Lemma~\ref{lem:ConstructionSymmetricSolution}; as 
		$F$ is arc-symmetric, $F^s=F$ and the operator costs do not change.
		For a solution $(f,y)$ of $LPP$ proceed analogously considering aggregated frequency values $F_a := \sum_{l\in L_a} f_l$  for all $a \in A.$		
	\end{proof}
\end{lemma}

\begin{lemma}
	If $(F,y)$ is an optimal solution of $ALPP$ with a symmetric passenger flow, then there exists a symmetric frequency plan $\tilde{F}$ such that $(\tilde{F},y)$ is feasible and optimal as well.
	
	The same holds for an optimal path-symmetric solution $(f,y)$ of $LPP.$
\end{lemma}
\begin{proof}
	Denote the aggregated passenger flow per arc by $Y_a = \sum_{p \in P_a} y_p$ for every $a \in A$. As $y$  is symmetric, $Y_a = Y_{\rho_z(a)}$ for any arc $a\in A$ and any $z \in \ZnZ$. The integer and capacity requirements further imply $F_a \geq \lceil Y_{a}/K \rceil$ for all $a \in A$, and
	by flow conservation at $CD$ and $P_i$,
	$$\sum_{i=0}^{n-1} F_{(SC_i,CD)} = \sum_{i=0}^{n-1} F_{(CD,SC_i)}
	\quad \text{ and } \quad 
	F_{(P_i, SC_i)} = F_{(SC_i, P_i)} \quad \text{ for } i \in \ZnZ.$$
	We define a symmetric frequency plan $\tilde{F}$ by giving the four values $\tilde{F}_P, \tilde{F}_{S+}, \tilde{F}_{S-}, \tilde{F}_C$ according to \Cref{lem:fourvalues}:
	\begin{align*}
	    \tilde{F}_P &:= F_{(P_0, SC_0)} , &
	    \tilde{F}_{S+} &:= \lceil Y_{(SC_0, SC_1)}/K \rceil , &
	    \tilde{F}_{S-} &:= \lceil Y_{(SC_0, SC_{n-1})}/K \rceil , &
	    \tilde{F}_{C} &:= \lceil Y_{(SC_0, CD)}/K \rceil.
	\end{align*}
	Then $\tilde{F}$ satisfies all capacity constraints on periphery arcs and on arcs connecting subcenters, and also on central district arcs as no trips originate in $CD$ and hence $Y_{(CD,SC_i)}\leq Y_{(SC_i,CD)}$.
	This shows that $(\tilde{F}, y)$ is feasible. For the optimality of $(\tilde{F}, y)$, we see that by definition of $\tilde{F}$, we have $\tilde{F}_a \leq F_a$ for all arcs except possibly for $a \in \{(CD,SC_i) \mid i \in \ZnZ\}$. However, since these arcs and their anti-parallel counterparts have the same lengths, flow conservation at $CD$ implies
	$$ \sum_{i=0}^{n-1} \tau_{(CD,SC_i)} F_{(CD,SC_i)} = \sum_{i=0}^{n-1} \tau_{(SC_i,CD)} F_{(SC_i, CD)} \geq \sum_{i=0}^{n-1} \tau_{(SC_i, CD)} \tilde{F}_C.$$
	Therefore, the cost of $(\tilde{F}, y)$ is at most the cost of $(F, y)$, but the latter was already optimal.
	
	For $LPP,$ consider aggregated frequencies $F_a := \sum_{l\in L_a} f_l$ to obtain arc-frequencies, symmetrize, and construct a symmetric line plan as described in Lemma~\ref{lem:symmF_implies_symmf}. 
\end{proof}

Our discussion can be summarized as follows.  
\begin{proposition}[Sufficient Condition for Symmetry]
\label{prop:sufficient}
An arc-symmetric, line-symmetric, or path-symmetric optimal solution is sufficient for the existence of a symmetric optimal solution.
\end{proposition}
In other words, if a symmetric optimal solution exists, then it is enough to find an optimal solution with a symmetric line plan, \emph{or} frequency plan, \emph{or} passenger flow.
Instances of the Parametric City can be analyzed with respect to the existence of optimal symmetric solutions by comparing the line optimization model $ALPP$ with its restriction $ALPP$ to arc-symmetric solutions
\begin{alignat*}{5}
&ALPP_{\mathcal{S}} &&= &ALPP& && \\
& &&  &s.t. \quad F_a &= F_{\rho_z(a)} \quad &&\forall z \in \ZnZ, a \in A.
\end{alignat*}
If the optimal objective values coincide, there is an arc-symmetric solution to $ALPP$, otherwise, there is a gap.

\section{Symmetry Gap}
\label{sec:symmetryGap}

For a feasible (mixed integer) linear program $P$, we denote the objective value of an optimal solution by $\OptVal(P)$\label{ass:OptVal}.

\begin{definition}[Symmetry Gap]
    For an instance of the line planning problem in the Parametric City, define the \emph{absolute symmetry gap} as
    $$\Gamma_{abs} := \OptVal(ALPP_{\mathcal{S}}) - \OptVal(ALPP)$$
    and the \emph{(relative) symmetry gap} as
    $$\Gamma := \frac{\OptVal(ALPP_{\mathcal{S}})- \OptVal(ALPP)}{\OptVal(ALPP)}$$
    if $ALPP$ is feasible;  $\Gamma_{abs} = 0$ and $\Gamma =0$ otherwise.
\end{definition}
    The symmetry gap is well defined, since we require the patronage $Y$ as well as demand and arc-lengths in the graph to be positive. Thus, there must exist an arc $a$ with frequency $F_a >0$ as well as a path $p$ with $y_p > 0,$ ensuring that $\OptVal(ALPP)>0.$  Note that $\Gamma_{abs} = 0$ and $\Gamma = 0$ hold if and only if there exists a symmetric optimal solution $(F,y)$ or no solution exists. 

\subsection{Bounds}
\begin{lemma}
    \label{lem:AbsSymGap}
    The absolute symmetry gap in the Parametric City is bounded by $$\Gamma_{abs} \leq 2 \mu  \left(\tau_{(SC_0,SC_1)}  + \tau_{(CD,SC_0)} \right) \,   (n-1)  =  2 \mu T \left(1+r_n \right) (n-1).$$
   
    \begin{proof}
        Let $(F,y)$ be the optimal solution to the unrestricted $ALPP$ and $(F^s,y^s)$ the symmetrized solution as in the proof of \Cref{lem:ConstructionSymmetricSolution}.
        %
        As user costs are invariant with respect to symmetrization,  see Lemma~\ref{lem:ConstructionSymmetricSolution}, the absolute symmetry gap is bounded by the difference in operator costs: 
        \begin{align*}
        \Gamma_{abs} &\leq cost(F^s,y^s) - cost(F,y) = \mu \sum_{a \in A} \tau_a \left(F_a^s-F_a\right).
        \end{align*}
        By construction of $F^s,$ 
        $\sum_{z=0}^{n-1} \left( F^s_{\rho_z(a)} - F_{\rho_z(a)} \right) \leq n-1$ for all $a \in A.$
        On the peripheral arcs however, by \Cref{lem:properties}, the frequencies $F$ and hence $F^s$ are always the same. Therefore, by using the explicit arc costs $\tau_a$ (cf. \Cref{ass:tau2}), we obtain an upper bound on the absolute gap:
        \begin{align*}
        \Gamma_{abs}
        &\leq \mu \sum_{a \in A} \tau_a \left(F_a^s-F_a\right)\\
        &\leq \mu \left(  \sum_{z=0 }^{n-1} T \, r_n (F^s_{\rho_z(SC_0,SC_1)} -F_{\rho_z(SC_0,SC_1)} +  F^s_{\rho_z(SC_1,SC_0)} -F_{\rho_z(SC_1,SC_0)})    \right. \\
        & \left. \quad \quad \quad
        + \sum_{z=0 }^{n-1} T \, (F^s_{\rho_z(SC_0,CD)} -F_{\rho_z(SC_0,CD)}+F^s_{\rho_z(CD,SC_0)} -F_{\rho_z(CD,SC_0)})  \right)\\
        & \leq 2 \mu \left(T \, r_n \, (n-1) + T \, (n-1) \right)  =  2 \mu  T  \left(1+r_n \right) (n-1). \qedhere
        \end{align*}
    \end{proof}
\end{lemma}


\begin{corollary}
    If $\mu = 0$ and $ALPP$ is feasible, then there exists a symmetric, optimal solution to the line planning problem for any instance of the Parametric City. 
\end{corollary}

To find an upper bound on the relative symmetry gap, we first need to determine a lower bound on the objective of $ALPP$: Such a lower bound, actually to the LP relaxation of $ALPP$, is provided by a certain minimum-cost flow  problem. Its optimal objective value can in turn be bounded from below by a term in parameters of the Parametric City that is independent from the number of zones and from the demand pattern.

 \begin{definition}[Uncapacitated Minimum-Cost Flow Problem]
     \label{UMCFP}
     \begin{mini!}|s|[2]<b>
         {y}{ \sum_{p \in P} \bar{c}_p \textcolor{black}{y_p} = cost_{UMCFP}(y)}
         {}
         {\textcolor{black} {(UMCFP)\quad} \nonumber}
         \addConstraint{\sum_{p \in P_{st}} \textcolor{black}{y_p}}{= d_{s,t} \quad \label{con:demand_sp}} {\forall (s,t) \in D}
         \addConstraint{\textcolor{black}{y_p} }{\geq 0 \label{con:paths_lb}}{\forall p \in P}
     \end{mini!}
     with $\bar{c}_p := \sum_{a \in p } \bar{c}_a$ and
     $$\bar{c}_a := \begin{cases}
     (2\mu/K +(1-\mu)) T \, g &\text{ for } a=(P_i,SC_i),\\
     (1-\mu) T \, g & \text{ for } a=(SC_i,P_i),\\
     (\mu/K + (1-\mu)) T \, r_n & \text { for } a = (SC_i, SC_{i\pm1}),\\
     (2\mu/K + (1-\mu)) T & \text{ for } a = (SC_i, CD),\\
     (1-\mu) T & \text{ for } a =(CD,SC_i).
     \end{cases}$$
 \end{definition}
 
 \begin{proposition}
     
     \label{prop:explicitSolutionToSP}
     The optimal objective value of the minimum-cost flow problem $UMCFP$ 
     is given by
     $$  \OptVal(UMCFP) = T{Y} \lambda(\alpha, \gamma),$$
     where
     \begin{align*}
     \lambda(\alpha,\gamma) &= \frac{ k_n(k_n+1)r_n+ 2 (n-2k_n+1)}{n-1} \left(\frac{\mu}{K} + (1-\mu)\right)  \, \left({a\gamma+(1-a) \tilde{\gamma}}\right) \\
     &\hspace{.5cm} + \left(\frac{2\mu}{K} + (1-\mu)\right) \left(a\alpha+(1-a)\tilde{\alpha} +g a \right),
     \end{align*}
     and $k_n:=\left\lfloor \frac{2}{r_n} \right\rfloor$.
    Furthermore, independent of the number of zones and the demand pattern, 
    $$(1+ga-a) \left( \left(2-\frac{2}{\pi} \right)\frac{\mu}{K} + 1-\mu\right) \; \leq \; 
    \lambda(\alpha,\gamma)
    \;\leq \; 4 (1+ga) \left(\frac{2\mu}{K} + 1-\mu\right).$$

 \end{proposition}
 \begin{proof} We can solve $UMCFP$ by determining the shortest path for each origin-destination pair. For each zone, these are the paths as depicted in Figure~\ref{fig:shortestpaths}. The passenger flow is determined by the demand. The details can be found in the appendix.
 \end{proof}
 
   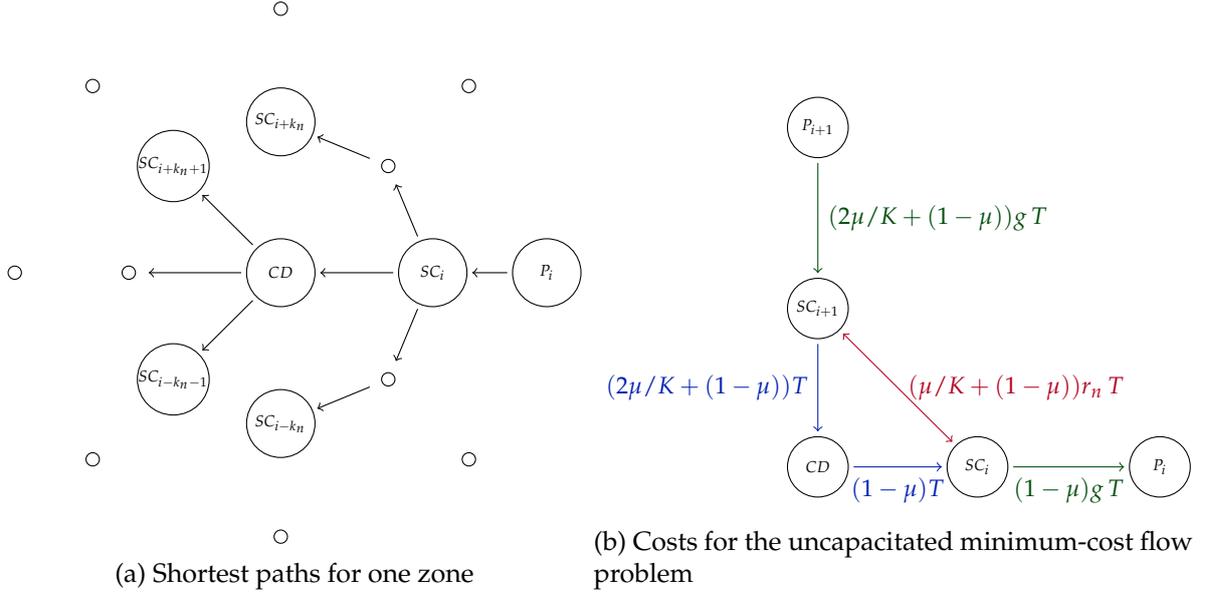
\begin{figure}
    \centering
    \begin{subfigure}[b]{0.495\textwidth}
    \centering
    \begin{tikzpicture}[myn/.style={draw,circle,fill=none, font=\footnotesize, outer sep=2pt, inner sep=0pt,minimum size=0.9cm},
dot/.style={circle, draw, fill=none, inner sep=0pt, minimum width=5pt, outer sep=5pt},]
\tikzmath{\n = 8; \nm1 = \n-1; \nm2 = \n-2; \T = 2; \gT = \T*(1+0.75);}
	\node[myn]
	(CD) at (0,0) {\tiny{$CD$}};
	
	\node[myn]
		(SC0) at ({0*360/\n}: \T cm) {\tiny{$SC_{i}$}};
					\node[dot]
		(SC1) at ({1*360/\n}: \T cm) {};
					\node[myn]
		(SC2) at ({2*360/\n}: \T cm) {\tiny{$SC_{i+k_n}$}};
					\node[myn]
		(SC3) at ({3*360/\n}: \T cm) {\tiny{$SC_{i+k_n+1}$}};
					\node[dot]
		(SC4) at ({4*360/\n}: \T cm) {};
					\node[myn]
		(SC5) at ({5*360/\n}: \T cm) {\tiny{$SC_{i-k_n-1}$}};
					\node[myn]
		(SC6) at ({6*360/\n}: \T cm) {\tiny{$SC_{i-k_n}$}};
		\node[dot]
		(SC7) at ({7*360/\n}: \T cm) {};

	\foreach \x in {1,...,\nm1}{%
		\node[dot]
		(P\x) at ({\x*360/\n}: \gT cm) {};
	}
	\node[myn]
		(P0) at ({0/\n}: \gT cm) {\tiny{$P_{i}$}};
	\draw[->] (P0) -- (SC0);
	\draw[->] (SC0) -- (CD);
	\draw[->] (CD) -- (SC3);
	\draw[->] (CD) -- (SC4);
	\draw[->] (CD) -- (SC5);

	\draw[->] (SC0) -- (SC1);
	\draw[->] (SC0) -- (SC7);
	\draw[->] (SC7) -- (SC6); 
	\draw[->] (SC1) -- (SC2);
\end{tikzpicture}
     \caption{Shortest paths for one zone}
    \label{fig:shortestpaths}
\end{subfigure}%
\begin{subfigure}[b]{0.495\textwidth}
\centering
\begin{tikzpicture}[myn/.style={draw,circle,fill=none, font=\footnotesize, outer sep=2pt, inner sep=0pt,minimum size=0.8cm}]
\tikzmath{\n = 8; \nm1 = \n-1; \nm2 = \n-2; \T = 2; \gT = \T*(1+0.75);}
			\node[myn]
			(CD) at (0,0) {\tiny{$CD$}};
            
            \node[myn]
            (SC0) at (0,2.1) {\tiny{$SC_{i+1}$}};
            \node[myn] 
            (P0) at (0,4.5) {\tiny{$P_{i+1}$}};
            \node[myn]
            (SC1) at (2.1,0) {\tiny{$SC_{i}$}};
             \node[myn] 
            (P1) at (4.5,0) {\tiny{$P_{i}$}};
            
            \path[->,draw,darkblue]
            (CD) edge node[below, font=\footnotesize, darkblue] {$(1-\mu)T$} (SC1);
            \path[->,draw,darkblue]
            (SC0) edge node[left,font=\footnotesize, darkblue] {$(2\mu/K + (1-\mu))T$} (CD);
            \path[<->,draw, darkred]
            (SC1) edge node[right,font=\footnotesize, darkred] {$(\mu/K + (1-\mu))r_n\,T$} (SC0);
              \path[->,draw, darkgreen]
            (SC1) edge node[below,font=\footnotesize, darkgreen] {$ (1-\mu)g\,T$} (P1);

              \path[->,draw,darkgreen]
            (P0) edge node[right,font=\footnotesize, darkgreen] {$(2\mu/K + (1-\mu))g\,T$} (SC0);

\end{tikzpicture}
    \caption{Costs for the uncapacitated minimum-cost flow problem}
     \label{fig:shortestpathvalues}
\end{subfigure}
\caption{Shortest path-related parameters}
\label{fig:shortestPathParameters}
\end{figure}


 
 \begin{lemma}
     \label{lem:lowerBound}
     For any instance of the Parametric City, the optimal objective value of the uncapacitated minimum-cost flow problem $(UMCFP)$ provides a lower bound for the line planning problem, i.e.,
     $$\OptVal(UMCFP) \leq \OptVal(ALPP).$$
    
 \end{lemma}
 \begin{proof}
     $UMCFP$ is equivalent to the following relaxation of $ALPP$: Relax the integrality constraints to non-negativity constraints, as well as the street-capacity constraints $(\Lambda \to \infty).$
     There are no travelers towards peripheries, the effective cost of a flow being directed via arc $(P_i, SC_i)$ is consequently twice as high, to compensate for the ``empty seats" in the opposite direction. Similarly 
     at the central node: Since there are no travelers originating from the central node, any flow using an outward arc from $CD$ must travelalong some arc $(SC_i, CD),$  any traveler remaining in the center ``blocks" a seat on the outgoing arcs. Consequently, one can assign twice the costs for the incoming arcs at $CD$ and set costs along the outgoing arcs to zero. 
     A detailed proof can be found in the appendix.
 \end{proof}

\begin{corollary}
    
	\label{cor:tightestBound}
For any instance of the Parametric City with fixed parameters $K \in \mathbb{N}_{>0}$ and $\mu \in [0,1],$ 
and $\lambda(\alpha, \gamma)$ as defined in Proposition~\ref{prop:explicitSolutionToSP}, the symmetry gap is at most 
$$\Gamma \leq \mu \frac{2(1+r_n)(n-1)}{Y \lambda(\alpha, \gamma)}.$$
\end{corollary}


\begin{corollary}
\label{cor:boundSP}
	The symmetry gap in the Parametric City is bounded by 
	$$\Gamma \leq \mu \frac{ 2 (1+r_n)(n-1)}{Y\left(\left(2-\frac{2}{\pi}\right) \frac{\mu}{K} + 1-\mu\right)(1+ga-a)},$$
	for fixed parameters $K \in \mathbb{N}_{>0}, \mu \in [0,1],$ and $\Lambda \geq 1,$ independent of the demand parameters $\alpha$ and $\gamma.$
\end{corollary}
	
For a realistic choice of parameters, the two terms $\left(\left(2-\frac{2}{\pi}\right) \frac{\mu}{K} + 1-\mu\right)$ and $\left(1+ga-a\right)$ carry little weight,
and the symmetry gap is dominated by the ratio $(n-1)/Y,$ which will be very small in real-life applications. 
%
%
In general, a better bound can be derived by establishing a lower bound on $\OptVal(MILP_A)$ which takes effect when the lower bound derived from the uncapacitated minimum-cost flow problem goes to zero. 


\begin{lemma}[Lower Bound on Operator Costs]
	\label{lem:operatorCostBound}
	For any instance of the Parametric City, the operator costs $\mu \sum_{a \in A} \tau_a F_a$ are at least $\mu T(2ng+2+(n-1)r_n).$
	\end{lemma}
	\begin{proof}	
		Suppose $F$ is a feasible solution. From Lemma~\ref{lem:properties} we know that $F_{a} = F_{a'} \geq 1$  for all arcs $a, a'$ incident with some periphery $P_i$. This means that 
		\begin{equation}
		\sum_{z = 0}^{n-1} F_{\rho_z(SC_0, P_0)} +F_{\rho_z(P_0, SC_0)}
		\geq 2n.
		\end{equation}Furthermore, 
		at least one of the incoming arcs as well as one of the outgoing arcs at $CD$ needs to be used and these numbers coincide by flow conservation, i.e.,  
		\begin{equation}
		\sum_{z = 0}^{n-1} F_{\rho_z(SC_0,CD)}= \sum_{z = 0}^{n-1} F_{\rho_z(CD, SC_0)} \geq 1
		\label{eq:c}
		\end{equation}
		Now consider the subgraph $W = (\bar{V}, \bar{A}), $ with $\bar{V} = \{CD, SC_0, SC_1,\dots, SC_{n-1}\}$ and $\bar{A} = \{a \in G[\bar{V}]: F_a >0 \}.$  Due to the positive demand and 
		flow conservation for $F,$ $W$ must be a strongly connected and there must be at least $n+1$ arcs in $W,$ so that 
		$$\sum_{z = 0}^{n-1} \left( F_{\rho_z(SC_0,SC_1)} +F_{\rho_z(SC_1,SC_0)} +F_{\rho_z(SC_0,CD)} +F_{\rho_z(CD,SC_0)} \right) \geq n+1.
		$$
		Equation~\eqref{eq:c} and the existence of an arc $(SC_j, CD)$ with $F_{(SC_j,CD)} \geq 1$ 
		imply
		$$\sum_{z = 0}^{n-1} \left(F_{\rho_z(SC_0,SC_1)} +F_{\rho_z(SC_1,SC_0)} \right) + 2\sum_{\substack{z=0\\z \neq j}}^{n-1} F_{\rho_z(SC_0,CD)}  \geq n-1.
		$$
		The operator costs therefore satisfy
		\begin{align*}
		\mu \sum_{a\in A} \tau_a F_a &= \mu \left(
		2\sum_{z=0}^{n-1} Tg F_{\rho_z(SC_0, P_0)}
		+ \sum_{z=0}^{n-1} Tr_n F_{\rho_z(SC_0, SC_1)}
		+ \sum_{z=0}^{n-1} Tr_n F_{\rho_z(SC_1, SC_0)} \right. \\ &\quad\quad \left. 
		+ 2 \sum_{\substack{z=0\\z \neq j}}^{n-1} T F_{\rho_z(SC_0,CD)}
		+ 2 T F_{(SC_j, CD)} \right) \\
		&\geq \mu ( 2n Tg  + (n-1) \min \{2T , Tr_n\} + 2T).
		\end{align*}
		Since $r_n = 2 \sin ( \pi/n) < 2 $ for all $n \geq 4$ the claim follows. 
		\end{proof}
\begin{corollary}[Upper Bound on $\Gamma$]

	For any instance of the Parametric City, and independent on the choice of parameters, except for $g,$ the symmetry gap can be bounded by the constant
	$$ \Gamma \leq  \frac{(1+\sqrt{2})}{g}.$$
	\begin{proof} Lemma~\ref{lem:operatorCostBound} and the fact that the objective value of $ALPP$ is a sum of operator and user costs imply $\OptVal(ALPP) \geq \mu \sum_{a \in A} \tau_a F_a .$ Using this and the fact that $r_n\leq \sqrt{2} <2$ for all $n\in \mathbb{N}_{\geq 4}$ gives rise to 
	the bound 
	
\begin{equation*}
\Gamma \leq \mu \frac{2T(1+r_n)(n-1)}{\OptVal(ALPP)} 
\leq \frac{2(1+r_n)}{ 2g+r_n} \leq \frac{(1+\sqrt{2})}{g}. \qedhere
\end{equation*}
	\end{proof}
\end{corollary}

 \begin{proposition}
 	\label{prop:relSymmetryGapBounds}
 	The relative symmetry gap $\Gamma$ in the Parametric City is bounded as follows:
 	 \begin{align*}
 	\Gamma \leq C_n(\alpha,\gamma) \leq C_n \leq \frac{(1+\sqrt{2})}{g}
 	\end{align*}
 	where 
 	\begin{align*}
 	C_n(\alpha,  \gamma) &:= \min \begin{cases}
 \frac{\mu (n-1)\, 2\, (1+r_n)}{ Y \rho(\alpha, \gamma)},\\
 	\frac{2(1+r_n)}{2g + r_n},
 	\end{cases}
 	\quad\text{\rm and}\quad
 	C_n &:= \min \begin{cases}
 	 \frac{\mu (n-1)\, 2 (1+r_n)}{Y\left(\left(2-\frac{2}{\pi}\right) \frac{\mu}{K} + 1-\mu\right)(1+ga-a)},\\
 	\frac{2(1+r_n)}{2g + r_n}.
 	\end{cases}
 	\end{align*}
 \end{proposition}

\begin{figure}[ht]
	\centering
	\begin{subfigure}{.5\textwidth}
		\centering
		\includegraphics[width=\linewidth]{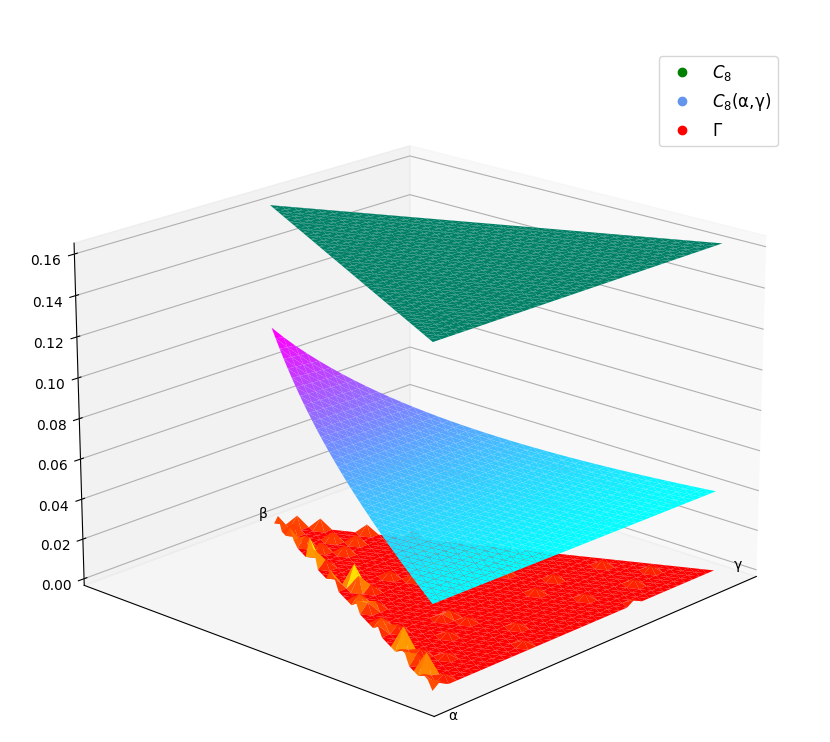}
	\end{subfigure}%
	\begin{subfigure}{.5\textwidth}
		\centering
		\includegraphics[width=\linewidth]{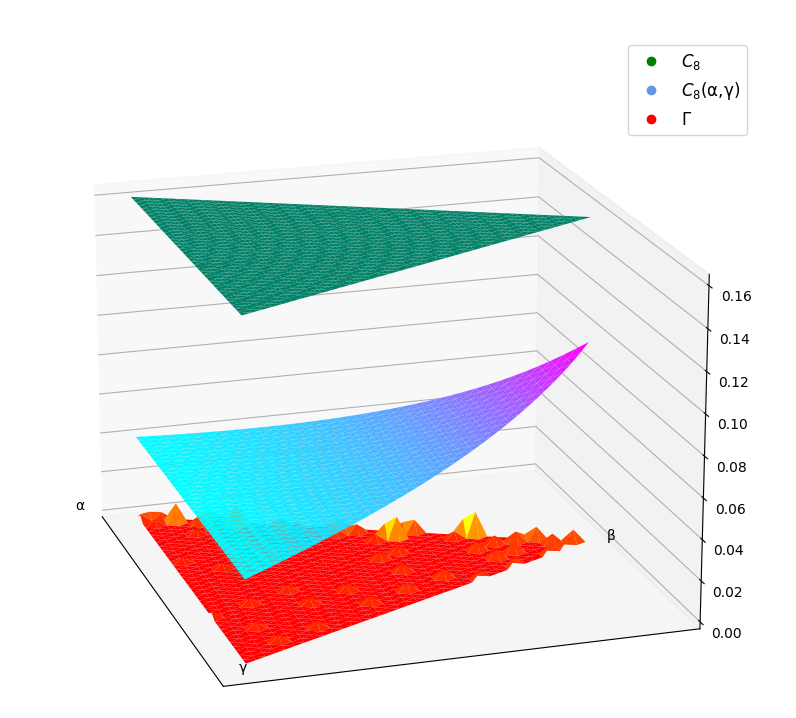}
	\end{subfigure}
	\caption[Symmetry gap -- theoretical bound vs. experimental results]{Theoretical symmetry gap for the Parametric City for $n=8, Y=24000, K=100, g=1/3, \mu = 1, a = 0.8$ in comparison to computational results (cf. \Cref{sec:computation})}
	\label{fig:symmetryGapExperiment}
\end{figure} 

For constant patronage $Y$, the bound $C_n$ increases roughly linearly with the number $n$ of zones, while for constant $n$, the influence of the term $1/Y$ leaves its mark for small $Y$.
Figure~\ref{fig:symmetryGapExperiment} plots the bounds $C_n(\alpha,  \gamma)$ and $C_n$ for (realistic) parameter choices in comparison to the actual gap $\Gamma$ obtained from computational
experiments. 
It is evident that $C_n(\alpha, \gamma)$ is considerably closer to the symmetry gap than $C_n.$ However, even $C_n$ is smaller than $0.16$ independent of the demand, less than $0.12$ as a function of $\alpha, \beta,$ and $\gamma,$ and less than $0.04$ for small $\beta.$ The real gap $\Gamma$ is considerably smaller -- in most cases, the symmetric solutions are optimal, and otherwise, the gap is below $1.22\%$ (see Figure~\ref{fig:3DErrorN8} and Table~\ref{tab:MaxError_n} for additional comparison).
This is good news for city planners.

\subsection{Approximation Algorithm}

%
We will show in this subsection that the symmetrization of an optimal solution of the line planning problem in the Parametric City actually gives rise to an approximation algorithm.
Using the observtion that symmetric solutions have at most four different arc-frequencies (\Cref{lem:fourvalues}) gives rise to an alternative formulation of $ALPP_S$ involving only three integer valued variables. Indeed, we can add the following restrictions to $ALPP_S$ w.l.o.g.:
\begin{align*}
F_a &= F_{\bar{a}}\quad &\text{ for all } a \in \bigcup_{z =0}^{n-1} \{(SC_z, CD)\} \text{ and } \bar{a} \in \bigcup_{z=0}^{n-1} \{(CD, SC_z)\},\\
F_a &= F_{\bar{a}} = \left\lceil \frac{Ya}{nK}\right\rceil \quad &\text{ for all } a \in \bigcup_{z=0}^{n-1} \{(SC_z, P_z)\} \text{ and } \bar{a} \in\bigcup_{z=0}^{n-1} \{(P_z, SC_z)\}.
\end{align*}
The latter is a simple assignment constraint that fixes the frequencies along the peripheral arcs to optimal vaules. 
The resulting model is as follows.

    \begin{definition} [3-Integer Symmetric Line Planning Problem]

    \label{MILP:Sym}
    \begin{mini!}|s|[2]
        {\kern-0.5cm F_{S+}, F_{S-}, F_C ,y\kern-0.5cm}{\mu \, (\kappa_{C} F_{C} + \kappa_S F_{S+} + \kappa_S F_{S-})   + (1-\mu) \sum_{p \in P} \tau_p \textcolor{black}{y_p} + \mu \, \kappa_P F_P}
        {}
        {\textcolor{black} {(ALPP3_{\mathcal{S}})\quad} \nonumber}
        \addConstraint{\sum_{p \in P_{s \to t}} \textcolor{black}{y_p}}{= d_{s,t} \quad \label{con:demand_s}} {\forall (s,t) \in D}
        \addConstraint{\sum_{p \in P_a} \textcolor{black}{y_p} - \textcolor{black}{F_{C}} K}{\leq 0 \label{con:capacity_s1} \quad}{\forall a \in  \{ (CD, SC_z), (SC_z, CD), z \in \ZnZ\}}
            \addConstraint{\sum_{p \in P_a} \textcolor{black}{y_p} - \textcolor{black}{F_{S+}} K}{\leq 0 \label{con:capacity_s2} \quad}{\forall a \in  \{ (SC_z, SC_{z+1}), z \in \ZnZ\}}
                \addConstraint{\sum_{p \in P_a} \textcolor{black}{y_p} - \textcolor{black}{F_{S-}} K}{\leq 0 \label{con:capacity_s3} \quad}{\forall a \in  \{ (SC_{z+1}, SC_z), , z \in \ZnZ\}}
        \addConstraint{\textcolor{black}{F_C, F_{S+}, F_{S-}}}{\leq \Lambda \label{con:arcCapacity_s}}{}
        \addConstraint{\textcolor{black}{F_C, F_{S+}, F_{S-}}}{\in \mathbb{N}\label{con:frequencies_s}}{}
        \addConstraint{\textcolor{black}{y_p} }{\geq 0 \label{con:paths_s}}{\forall p \in P}
    \end{mini!}
 with the constant $F_P = \left\lceil \frac{Ya}{nK} \right\rceil \leq \Lambda$ and cost parameters
    \begin{align*}
        \kappa_C &= 2 n \tau_{(SC_0, CD)} = 2n\,T\\
        \kappa_{S} &= n \tau_{(SC_0, SC_1)} = n \,r_n T\\
        \kappa_P &= 2 n \tau_{(SC_0, P_0)} = 2n\,Tg.\\
    \end{align*}

   \end{definition}

    \begin{proposition}
    \label{prop:polynomialTimeALPP}
    The symmetric line planning problem $ALPP_{\mathcal{S}}$ is solvable in polynomial time.
    \end{proposition}
    \begin{proof}
        Mixed-integer programming problems with a fixed number of variables are solvable in polynomial time, as was proven by \citet{Lenstra}. 
    \end{proof}

\begin{definition}[Approximation Algorithm $(LPA)$]~
	
	\begin{tabular}{rl}
		Input: & Instance of a Parametric City \\
		Output: &Best symmetric solution $(f,y)$ to $LPP$ if feasible, $\emptyset$ otherwise.\\
		1. & Solve $ALPP_{\mathcal{S}}.$ If feasible, retrieve $(F,y)$, otherwise return $\emptyset$. \\
		2. & Use $(F, y)$ to receive $(f,y)$ as in Lemma~\ref{lem:symmF_implies_symmf}. Return $(f, y)$.
	\end{tabular}
\end{definition}
\begin{theorem}
	$LPA$ is a $\kappa$-factor approximation algorithm for the line planning problem in the Parametric City for fixed $g$, where $\kappa = 1+\frac{(1+\sqrt{2})}{g}.$
	\begin{proof}
	LPA is a polynomial time algorithm, as $ALPP_{\mathcal{S}}$ can be solved in polynomial time, cf. Proposition~\ref{prop:polynomialTimeALPP}, and we can assign frequencies to lines as described in Lemma~\ref{lem:symmF_implies_symmf} in $O(n)$. 
		Let $PC_g$ be the set of all instances of Parametric Cities for a given value of the parameter $g.$ Due to the the upper bound on the symmetry gap from Proposition~\ref{prop:relSymmetryGapBounds}, the performance ratio $\frac{\OptVal(ALPP_{\mathcal{S}})}{\OptVal(ALPP)}$ is bounded by
		
		\begin{equation*}
		\sup_{pc \in PC_g}\frac{\OptVal(ALPP_{\mathcal{S}})}{\OptVal(ALPP)} = \max_{pc \in PC_g} \; (1+\Gamma) \leq 1+\frac{1+\sqrt{2}}{g}.\qedhere
		\end{equation*}
		
	\end{proof}
\end{theorem}

We now give a family of instances that shows that the symmetry gap is indeed unbounded in the worst case. 
\begin{proposition}
	Let $PC$ be the set of all instances of the Parametric City. Then the worst case performance ratio of the approximation algorithm $LPA$ is unbounded, i.e.,
	$$\sup_{pc \in PC} \frac{\OptVal(ALPP_{\mathcal{S}})}{\OptVal(ALPP)} = \infty.$$
	\end{proposition}
	\begin{proof}
		Consider a Parametric City with parameters $Y, \alpha, \gamma, a, T.$ Set the scalarization parameter to $\mu = 1$  such that only operator costs are considered. Further, choose any $n \in \mathbb{N}_{\geq 4}$ and a large $K,$ e.g., $K=Y,$ which means that the total patronage can fit into a single vehicle. Lastly, choose a small $g$, say $g=1/n.$ 
		
		As all passengers fit into one vehicle, the frequencies of an optimal solution of $ALPP$ and $ALPP_{\mathcal{S}}$ on any arc are either $0$ or $1.$ 
		Consider the following line plan: Set $F_{(SC_0,CD)} = 1$  and $F_{(CD, SC_1)} = 1,$ all other arcs incident to $CD$ are set to zero. Further, assign $F_{\rho_z(SC_0,SC_1)} = 1$ for all $z \in \ZnZ$, $z \neq 0$, i.e., use all inter-subcenter arcs in counter clockwise direction except for $(SC_0,SC_1).$ Lastly, use all peripheral arcs, that is $F_a = 1$ for all $a \in \{ (P_z,SC_z), (SC_z,P_z), z\in \ZnZ\}.$ Any other arcs set to zero. 
		This results in the representative line plan in Figure~\ref{fig:extrememilpa}.  
		%
		The cost function reduces to pure operator costs since $\mu = 1.$ The objective value is therefore $$cost(F,y) = 2ng+ (n-1)r_n +2$$ which is exactly the lower bound for operator costs of Lemma~\ref{lem:operatorCostBound} and must consequently be optimal. Up to rotation and reflection, this solution is the same for all $n\geq 4.$

\begin{figure}
\centering
\begin{subfigure}{.5\textwidth}
\centering
\begin{tikzpicture}[myn/.style={draw,circle,fill=none, font=\footnotesize, outer sep=2pt, inner sep=0pt,minimum size=0.4cm},
dot/.style={circle, draw, fill=black, inner sep=0pt, minimum width=3pt, outer sep=5pt},]
\tikzmath{\n = 8; \nm1 = \n-1; \nm2 = \n-2; \T = 2; \gT = \T*(1+1/2);}
			\node[myn]
			(CD) at (0,0) {\tiny{$CD$}};
			\foreach \x in {0,...,\nm1}{
				\node[myn]
				(SC\x) at ({\x*360/\n}: \T cm) {\tiny{$SC_{\x}$}};
			}
			
			\foreach \x in {0,...,\nm1}{%
				\node[myn]
				(P\x) at ({\x*360/\n}: \gT cm) {\tiny{$P_{\x}$}};
			}
			\foreach \x in {0,...,\nm1}{%
					
				\DoubleLine{SC\x}{P\x}{<-,darkblue}{}{-> ,darkblue}{};
				
				}
				
			\foreach \x in {1,...,\nm1}{%
			    \pgfmathtruncatemacro {\y}{mod(round (1+\x),\n)}
			    	\draw[darkblue,->] (SC\x) -- (SC\y);
			    	}

			\draw[darkblue,->] (SC0) -- (CD); 
			\draw[darkblue,->] (CD) to (SC1);
        
\end{tikzpicture}
    \caption{$ALPP$}
    \label{fig:extrememilpa}
\end{subfigure}%
\begin{subfigure}{.5\textwidth} 
\centering
\begin{tikzpicture}[myn/.style={draw,circle,fill=none, font=\footnotesize, outer sep=2pt, inner sep=0pt,minimum size=0.4cm},
dot/.style={circle, draw, fill=black, inner sep=0pt, minimum width=3pt, outer sep=5pt},]
\tikzmath{\n = 8; \nm1 = \n-1; \nm2 = \n-2; \T = 2; \gT = \T*(1+1/2);}
			\node[myn]
			(CD) at (0,0) {\tiny{$CD$}};
			\foreach \x in {0,...,\nm1}{
				\node[myn]
				(SC\x) at ({\x*360/\n}: \T cm) {\tiny{$SC_{\x}$}};
			}
			
			\foreach \x in {0,...,\nm1}{%
				\node[myn]
				(P\x) at ({\x*360/\n}: \gT cm) {\tiny{$P_{\x}$}};
			}
			\foreach \x in {0,...,\nm1}{%
					
				\DoubleLine{SC\x}{P\x}{<-,darkblue}{}{-> ,darkblue}{};
				\DoubleLine{SC\x}{CD} {<-,darkblue}{}{-> ,darkblue}{};
				}

\end{tikzpicture}
    \caption{$ALPP_{\mathcal{S}}$}
    \label{fig:extrememilps}
\end{subfigure}%
\caption{Optimal frequency plans for the extreme example: $n\geq4, K = Y, g=1/n$ for any choice of $Y, a, \alpha, \gamma, T$; blue arcs indicate Frequency $1.$}
\end{figure}
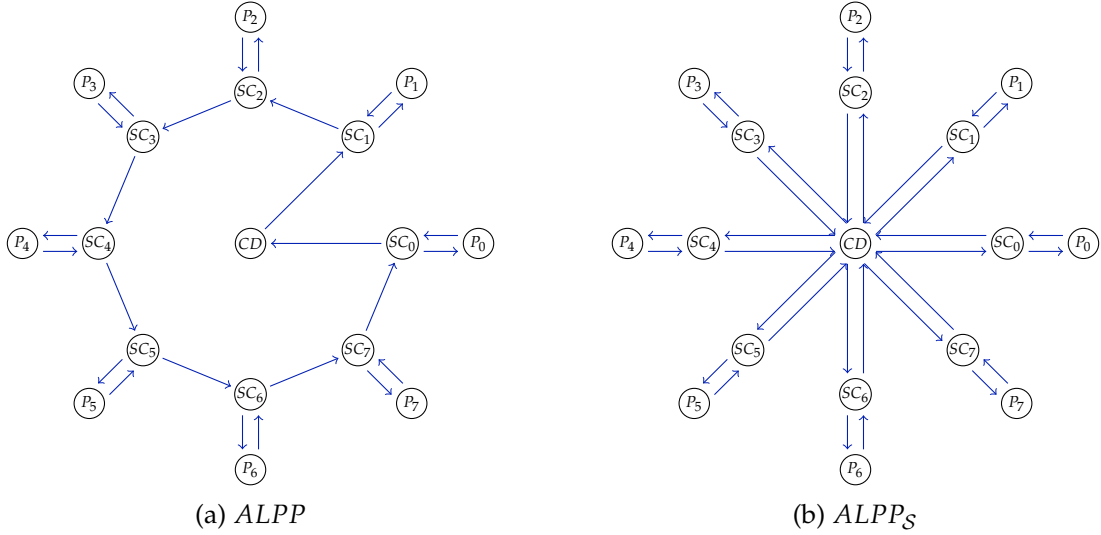	

		 From previous considerations we already know that in a solution to $ALPP_{\mathcal S}$, all arcs incident to the center must have positive and, in particular, equal frequency. As a minimal requirement for a feasible solution we thus have $F_P\geq1$ and $F_C \geq 1.$ Thus, assign frequency $1$ to all arcs incident to a periphery or the central business district, set all other arcs to zero. This line plan gives rise to a feasible solution -- see Figure~\ref{fig:extrememilps} for reference: Each node can be reached from any other node. Again, since the vehicle capacity is large enough to fit all passengers of the whole system, we can find admissible passenger paths. 
		 The associated objective value is $$\OptVal(ALPP_{\mathcal{S}}) = T(2ng+2n).$$
		Thus, for the chosen instance of the Parametric City, we have the ratio \begin{align*}
		\frac{\OptVal(ALPP_{\mathcal{S}})}{\OptVal(ALPP)} = \frac{T(2ng+2n)}{T(2ng+(n-1)r_n+2)} =\frac{2ng+2n}{2ng+(n-1)r_n+2}
		\end{align*}
		As we chose $g =1/n$ and using the fact that $r_n(n-1) \leq 2\pi,$ we find the lower bound 
		\begin{align*}
		\frac{2ng+2n}{2ng+(n-1)r_n+2} = \frac{2+2n}{2+(n-1)r_n+2} \geq \frac{2+2n}{2+2 \pi +2} = \frac{1+n}{2+\pi}. 
		\end{align*}
		
		As $n$ was chosen freely, this lower bound becomes arbitrarily large for increasing $n.$ The maximal performance ratio over \emph{all} instances of the Parametric City hence goes to infinity. 
	\end{proof}

\section{Computational Study}
\label{sec:computation}

How important are asymmetric solutions and how much better are they really? We study this question in this section computationally by considering a large number of instances of $ALPP$ with varying inputs. In all instances, we choose a total patronage of $Y=24000,$ from which $a=0.8$ originates in the peripheries. The distance between the subcenters and the central business district is $T= 30$ and the peripheries are at distance $gT$ from the subcenters, with $g=1/3.$ We vary the rest of the parameters: the vehicle capacity $K,$  the number of zones $n$, and the scalarization parameter $\mu.$ We then compare the results in maps dependent on $\alpha, \beta, \gamma$. 
%
We solve both $ALPP$ and $ALPP_{\mathcal{S}}$ to obtain the classification of each instance of the Parametric City for all $\alpha, \gamma \in [0.025,0.95]$ such that $\alpha+\beta +\gamma =1$ with a step size of $0.025.$ We choose the MIP solver Gurobi~9~\citep{gurobi} and consider a solution to be optimal, when the relative MIP optimality gap is below a tolerance of $10^{-4}.$ 


\subsection{Sensitivity Analysis}
\label{subsec:sensitivity}

\subsubsection{Influence of the Scalarization Parameter $\mu$}

The first set of computational experiments considers the influence of the parameter $\mu$. We evaluated the Parametric City for $n = 8$ for four different choices of $\mu,$ namely for the two extreme cases of $\mu = 0$ and $\mu=1,$ as well as $\mu = 0.5$ and $0.75,$ see Figure~\ref{fig:typeMaps_depMU}.

\begin{figure}[ht]
 \centering
 \begin{subfigure}{0.49\textwidth}
 	\includegraphics[width=\linewidth]{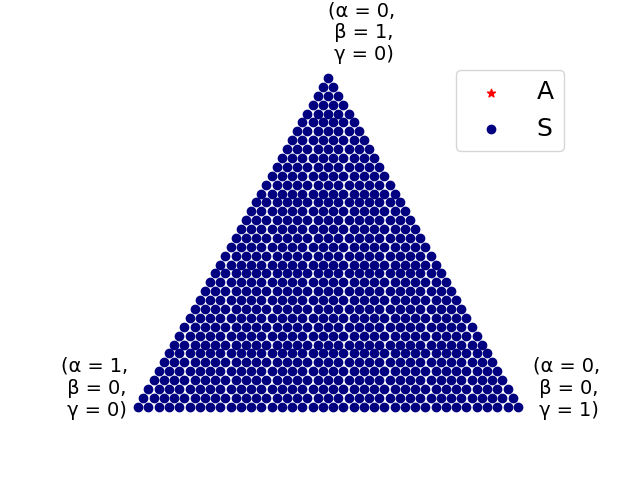}
 	\caption{$\mu = 0$}
 	\label{fig:mu0}
 \end{subfigure}%
 ~
 \begin{subfigure}{0.49\textwidth}
 	\includegraphics[width=\linewidth]{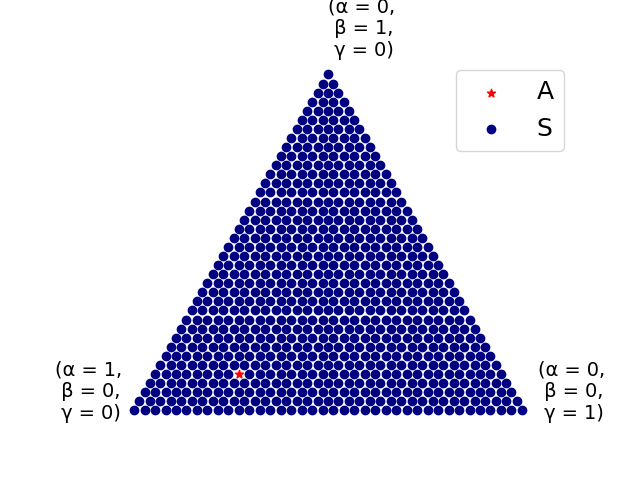}
 	\caption{$\mu = 0.5$}
 \end{subfigure}
 \begin{subfigure}{0.49\textwidth}
 	\includegraphics[width=\linewidth]{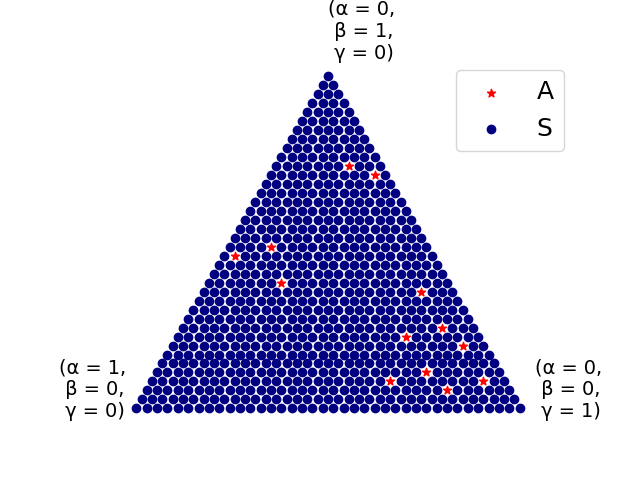}
 	\caption{$\mu =0.75$}
 \end{subfigure}%
 ~
 \begin{subfigure}{0.49\textwidth}
 	\includegraphics[width=\linewidth]{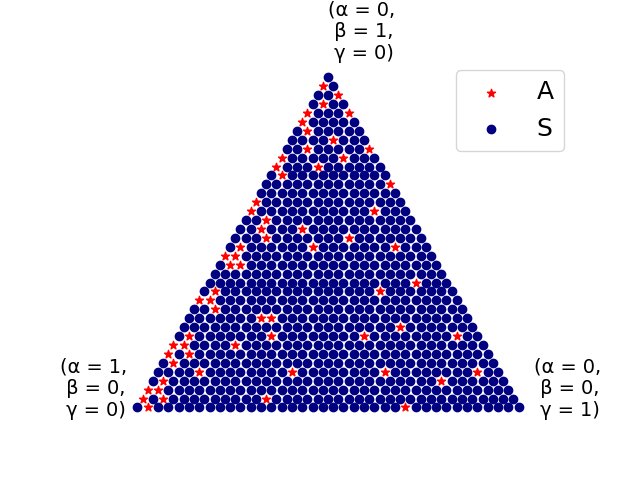}
 	\caption{$\mu = 1$}
 	\label{fig:mu1n8K100}
 \end{subfigure}
 \caption[Classification of optimal solutions with respect to $\mu$]{Optimal solutions (A/S: a-/symmetric) for $n=8, K=100$ with respect to $\mu$}
 \label{fig:typeMaps_depMU}
\end{figure}


The number of symmetric solutions greatly outweighs the number of asymmetric ones. For $\mu = 0$ there are none, for $\mu= 0.5$ there is only one, for $\mu = 0.75$ there $13$, and for $\mu = 1$ approximately $8.5\%$ are asymmetric solutions. Asymmetric optimal line plans are hence rare, but not singular or exceptional. Except for $\mu=0,$ i.e., when operator costs are ignored, one cannot assume that optimal solutions are symmetric. 
Indeed, the number of asymmetric cases increases with $\mu.$ This is not particularly surprising: The larger $\mu,$ the more focus is on the operators cost and the less expensive is it to reroute people from their preferred, shortest path to some longer detour in order to decrease the frequency along some arcs.

For the following computations, we fix the scalarization parameter to $\mu=1,$ because this choice produces  the most diverse results with respect to symmetry. This setting is also relevant from a practical point of view as well: More often than not, the main objective 
is the minimization of operating costs while providing enough service to cover demand.

\subsubsection{Influence of the Number of Zones $n$}

To examine how the number of zones affects the type of solutions, we compared them for $n=\{4,5,6,7,8\}.$ The results can be seen in Figure~\ref{fig:typeMaps_depN} and Figure~\ref{fig:mu1n8K100} for $n = 8$.
  \begin{figure}[ht]
	\centering
	\begin{subfigure}{0.49\textwidth}
		\includegraphics[width=\linewidth]{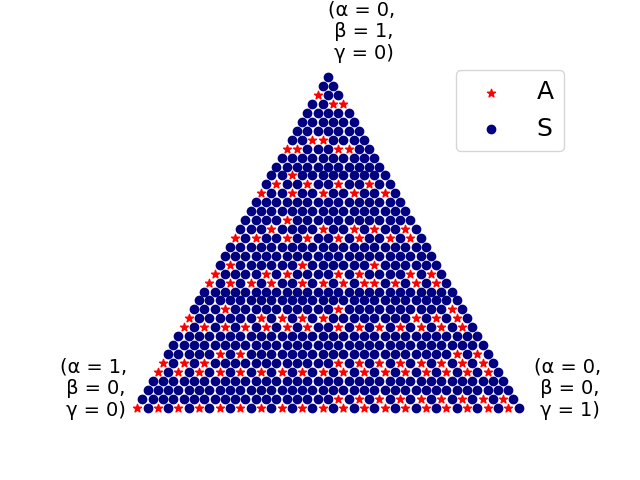}
		\caption{n=4}
		\label{fig:n4ref}
	\end{subfigure}%
	~
	\begin{subfigure}{0.49\textwidth}
		\includegraphics[width=\linewidth]{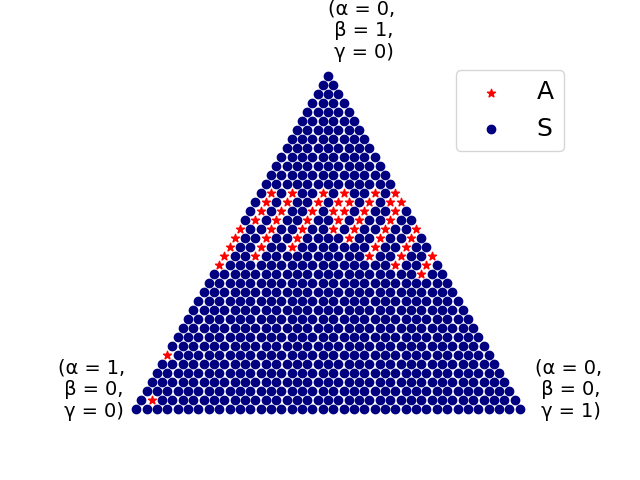}
		\caption{n=5}
		\label{fig:n5ref}
	\end{subfigure}
	\begin{subfigure}{0.49\textwidth}
		\includegraphics[width=\linewidth]{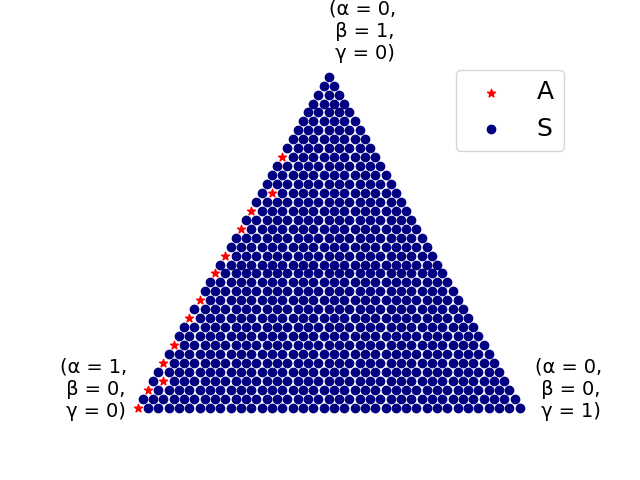}
		\caption{n=6}
		\label{fig:n6ref}
	\end{subfigure}%
	~
	\begin{subfigure}{0.49\textwidth}
		\includegraphics[width=\linewidth]{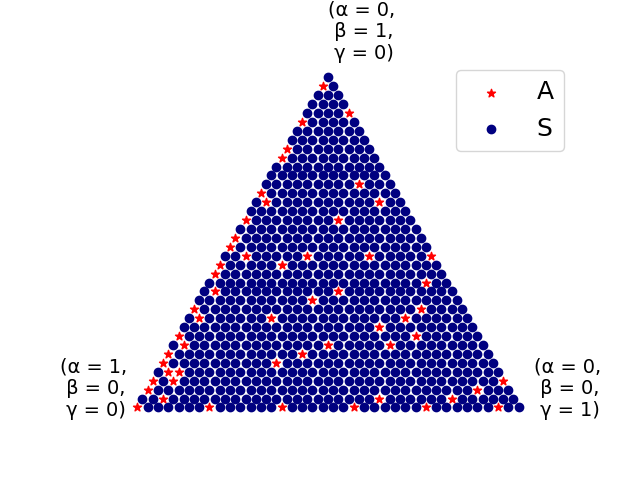}
		\caption{n=7}
		\label{fig:n7ref}
	\end{subfigure}
	\caption[Classification of optimal solutions in dependence of $n$]{Classification of optimal solutions for $K=100, \mu =1$ in dependence of $n$}
	\label{fig:typeMaps_depN}
\end{figure}

The results are somewhat surprising. There are asymmetric solutions for all choices of $n$, even for the ``maximally symmetric'' case $n=6$, in which the distance between two subcenters is the same as between a subcenter and the central business district. The Parametric City with six zones is the only one in which the shortest path between two vertices is not unique, which gives rise to more optimal solutions. One could expect more of them to be symmetric. And indeed, the number of asymmetric solutions is reduced in comparison to other choices of $n,$ but they do exist.

Another noticeable aspect is the accumulation of asymmetric cases on the left-hand side of the parameter triangle. This corresponds to low values of $\gamma.$, i.e., the fraction of total trips from a periphery to \emph{all} other subcenters. The demand between a periphery and a different subcenter is scaled by $\frac{1}{n-1}$ (cf. Table~\ref{tab:OD}), and is thus even smaller. As a consequence, we have more occurrences in which some of these few passengers can be rerouted to other arcs, which still have available capacity, such that the frequency along the direct path can be reduced. This correlates with an observation we will make when investigating  the influence of the vehicle capacity $K.$

\subsubsection{Influence of the Vehicle Capacity $K$}
We now study the impact of the vehicle capacity $K.$
For our sample size, it turns
out that the absolute number of asymmetric solutions increases with $K.$ For example, for $K = 50$ there are approximately 6.3\% asymmetric solutions,
while for $K = 100$ we find 8.5\%, and for $K = 150$ even 11.2\% are asymmetric. 
This can be explained in the following way: When the vehicle size is larger, it is more likely that there is still some capacity available to accommodate more passengers by detouring them via non-direct paths and thus reduce the frequency elsewhere. This is a similar phenomenon as for small values  of $\gamma$.

	\begin{figure}[ht]
 		\centering
  \begin{subfigure}[b]{0.49\textwidth}
 	\includegraphics[width=\linewidth]{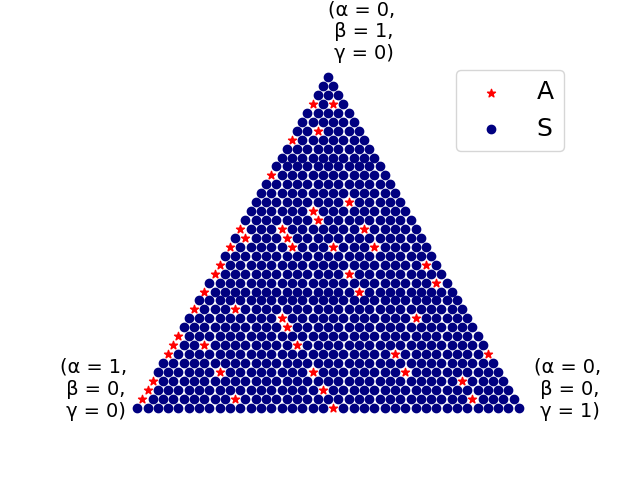}
 	\caption{K=50}
 	\label{fig:K50}
 \end{subfigure}%
~
 \begin{subfigure}[b]{0.49\textwidth}
 	\includegraphics[width=\linewidth]{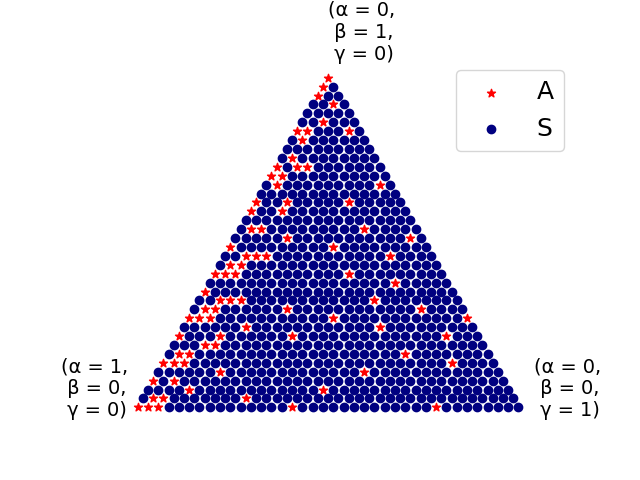}
 	\caption{K=150}
 	\label{fig:K150}
 \end{subfigure}
 \caption[Classification of optimal solutions in dependence of $K$]{Classification of optimal solutions for $n=8, \mu =1$ for in dependence of parameter $K$}
 \label{fig:typeMaps_depK}
\end{figure}

\subsection{Qualitative Analysis}
\label{subsec:qualitative}


This section studies the influence of the solution type on the overall costs. 
We plot the optimal values of $ALPP_{\mathcal{S}}$ and $ALPP$ in a $3D$-plot depending on the demand parameters $\alpha, \beta$ and $\gamma.$ The cost of $ALPP$ is marked with red dots, the cost of $ALPP_{\mathcal{S}}$ with blue ones. If the costs are equal, the red and blue dots are merged into a purple one for better readability. We include the markers $\alpha =1 , \beta =1$ and $\gamma =1$ to indicate the coordinates: At $\alpha = 1$ we have the coordinates $(\alpha, \beta, \gamma) =(1,0,0)$ etc.
It turns out the absolute symmetry gap is so small that the difference is hard to see with the bare eye.
We therefore rather plot relative symmetry gaps 
as in Figure~\ref{fig:3DErrorN8}\footnote{The red markers in the graph are included just for orientation: They are put at $0$ for depth perception such that the deviation of symmetric solutions from the optimum can be easier ascribed to the corresponding $(\alpha,\beta, \gamma)$ tuple.}; %
there, the relative gap is less than $0.012,$ i.e., symmetric solutions deviate from the optimal solutions by no more than $1.2\%.$


\begin{figure}[ht]
	\centering
	\includegraphics[width=.8\linewidth]{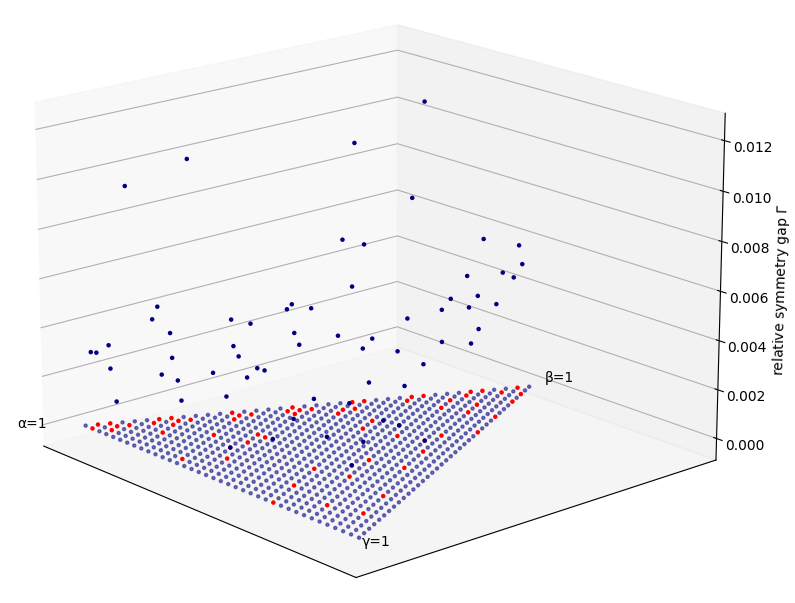}
	\caption[Relative gap of symmetric solutions]{Relative symmetry gap 
	of solutions for $n = 8, K = 100, \mu = 1$}
	\label{fig:3DErrorN8}
\end{figure}

The relative symmetry gaps look very similar for all instances of the previous section, such that we discuss Figure~\ref{fig:3DErrorN8} with $n=8, K=100, \mu = 1$ as a general representative. Varying one parameter and fixing all else, we make the following observations: 

\begin{itemize}
\item Total costs are lowest in the right corner for all data sets, for large $\beta$, i.e., when most people travel from the periphery only to their own subcenter. This is an expected result, since the frequency on any arc $(P_i, SC_i)$ is $F_{(P_i,  SC_i)} = \left \lceil \frac{Y a}{n K} \right \rceil$ by Lemma~\ref{lem:properties} -- independent of $\alpha, \beta, \gamma.$ Consequently, if $\beta$ is large, most of the passengers with origin in the peripheries will travel merely over one peripheral arc, while for smaller $\beta$ they will use other arcs as well, contributing to a larger total number of frequencies.  

\item The symmetry gap seems to increase with $K$: Table~\ref{tab:MaxError_K} records the maximal gap over all tested $\alpha, \beta, \gamma$ for each Parametric City with $\mu =1$ and for $n =6$ and $n = 8$ for different choices of $K.$
\item In comparison to $n$ we observe maximal gap values according to Table~\ref{tab:MaxError_n}: They increase with $n$ until $n=6$ and decrease afterwards.
	\item For increasing $\mu,$ the total costs decrease on average, as well as the minimum and maximum values. The latter correlates with the general behavior of the lower bound gained from the cost of the uncapacitated minimum-cost flow problem. In contrast, the maximum relative gap increases with $\mu$ in our experiments, as can be seen in Table~\ref{tab:MaxError_mu}. 

\end{itemize}
\begin{table}[ht]
	\begin{subtable}{\textwidth}
	\begin{center}
		\begin{tabular}{|l|ccc|ccc|}
			\hline
			&\multicolumn{3}{c|}{$n = 6$}&\multicolumn{3}{c|}{$n=8$}\\
			\hline
			$K$&50&100&150& 50&100&150\\
			\hline
			max. gap&0.82\%&1.66\% & 2.38\% & 0.51\% & 1.22\% & 3.21\% \\
			\hline
		\end{tabular}
	\caption[Maximal symmetry gap in dependence of $K$]{In dependence of $K$  for fixed $(\mu = 1)$}
	\label{tab:MaxError_K}	
	\end{center}
\end{subtable}\vspace{0.5cm}
\begin{subtable}{\textwidth}
	\begin{center}
		\begin{tabular}{|l| c c c c c|}
			\hline
			$n$& 4 & 5 & 6 & 7 & 8\\
			\hline
			max. gap & 0.27\% & 0.58\% & 1.65\% & 1.25\% & 1.22\%\\
			\hline
		\end{tabular}	
	\caption[Maximal experimental symmetry gap in dependence of $n$]{In dependence of $n$ with $(K=100,\mu =1)$}
	\label{tab:MaxError_n}
	\end{center}
\end{subtable}\vspace{0.5cm}
  \begin{subtable}{\textwidth}
 	\begin{center}
 		\begin{tabular}{|l| r r r r|}
 			\hline
 			&$\mu=0$ & $\mu = 0.5$ & $\mu\approx 0.878$ & $\mu=1$\\
 			\hline
 			average cost & 855 477.3& 434 727.5&116 055.4 &17 943.7 \\
 			min. cost &409 138.0 &   208 296.3&56 167.2 & 6 974.8\\
 			max. cost & 1 250 409.4&632 742.0 &165 634.4 & 13 181.0\\
 			\hline
 			max. gap & 0.0\% & 0.000611\% & 0.1\% & 1.22\%\\
 			\hline
 		\end{tabular}
 	\caption[Maximal experimental symmetry gap in dependence of $\mu$]{In dependence of $\mu$ $(K =100, n =8)$}
 	\label{tab:MaxError_mu}
 	\end{center}
 \end{subtable}
	\caption[Maximal symmetry gap w.r.t. symmetry in computational results]{Maximal gap of symmetric solutions of computational results dependent on different parameter choices}
	\label{tab:MaxError}
\end{table} 

\subsection{Comparison to Fielbaum et al.}
\label{subsec:fielbaum}

The next batch of computations compares our results to those of Fielbaum et al. In multiple publications \citep{FIELBAUM2016298, FielbaumTransitLineStructures,Fielbaum:Thesis}, they evaluate the performance of a transportation network with respect to the Value of the Resources Consumed (VRC), a socio-economic cost function first introduced and analyzed by \citet{Jansson1980} and later adjusted to include operator costs and vehicle sizes by \citet{JaraDiazGschwender}. The $VRC$ accounts mainly for distance-based travel times from both an operator's and a user's perspective, but it also measures delay due to boarding or alighting, as well as waiting times at vehicle arrivals. For a single line, the $VRC$ can be computed by a formula. The article by \citet{FIELBAUM2016298} generalizes the $VRC$ to a multiple-line model and applies it to the Parametric City by numerical minimization.

The transportation network is obtained by solving a variation of the model $LPP$ with an objective that is supposed to approximate VCR. In solving the line planning problem, Fielbaum et al. use heuristc approaches and make the general assumption that one can take ``advantage of the symmetry of the city" \cite[p. 12]{Fielbaum:Thesis}. They optimize over one zone only and replicate the optimal solution in the other zones.

The line planning problem that Fielbaum et al. consider 
differs slightly from ours. 
 Firstly, the line pool differs from ours: They use only ``lines that are shor\-test-path between their origins and destinations'' \citep[p. 14]{Fielbaum:Thesis} to reduce its size. As was described in \Cref{sec:linePlanning} our arc-based formulation of the line planning problem allows us to recreate any line corresponding to a directed cycle in the graph. We therefore consider a larger line pool.
Secondly,
they use fixed values to balance operator costs $c_o = 10.65\$/h$ against user costs $p_v =1.48 \$/h.$
In their notation the objective is 
 $$`` \sum_{R \in U} p_v \,t_R \, y_R + \sum_{l \in L} c_0 \, B_l  \text{''}  \quad \text{\cite[p. 87]{Fielbaum:Thesis}.}$$
 This objective can be directly translated into our notation:  $U$ is the set of paths, so it corresponds to our set $P$, $L$ is the set of lines, $t_R$ is the time spent traveling along path $R\in U,$ which means that it corresponds to our parameter $\tau_p$ for $p \in P.$ The parameter $B_l$  is the fleet size of line $l.$ As they define frequency as ``total fleet size divided by cycle time'' \cite[p. 6]{Fielbaum:Thesis}, this corresponds to $B_l = c_l f_l.$ Therefore, Fielbaum's cost function in our notation can be written as
 $$(c_0+p_v) \left(  \frac{p_v}{c_0+p_v} \sum_{p \in P} \tau_p y_p + \frac{c_0}{c_0+p_v} \sum_{l \in L} \tau_l f_l \right).$$
 
  By choosing $\mu = \frac{c_0}{c_0+p_v}$ this objective corresponds exactly to the one of $LPP$ up to the constant factor $c_0+p_v$. A constant factor in a linear objective function has no influence on the optimal solution $(f,y)$ of the problem. Therefore, by solving $ALPP$ with $\mu = \frac{c_0}{c_0+p_v} \approx 0.878$ we consider a comparable objective function.


 \begin{figure}[hbtp]
 \centering
 \begin{subfigure}[b]{0.46\textwidth}
 	\includegraphics[width=\linewidth]{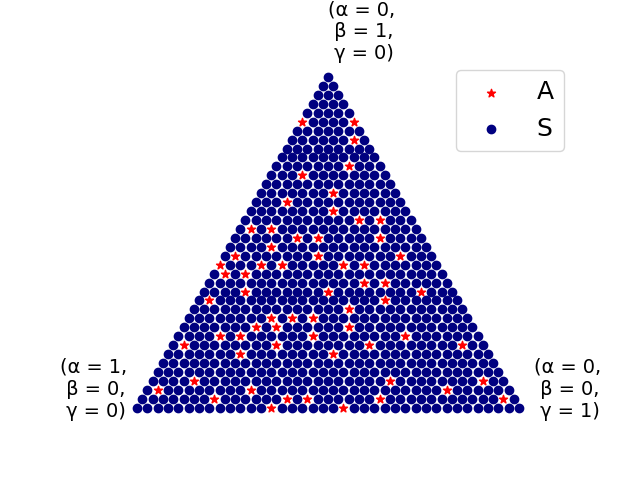}
 \caption[Classification of optimal solutions in comparison to Fielbaum]{Classification of optimal solutions for $n=8, Y=24000, g=1/3, a=0.8$ and $ \mu \approx 0.878$ corresponding to Fielbaum's results}
 \label{fig:typeMapFielbaum}
 \end{subfigure} \hspace{1em}
 \begin{subfigure}[b]{0.46\textwidth}
     	\includegraphics[width=.9\linewidth]{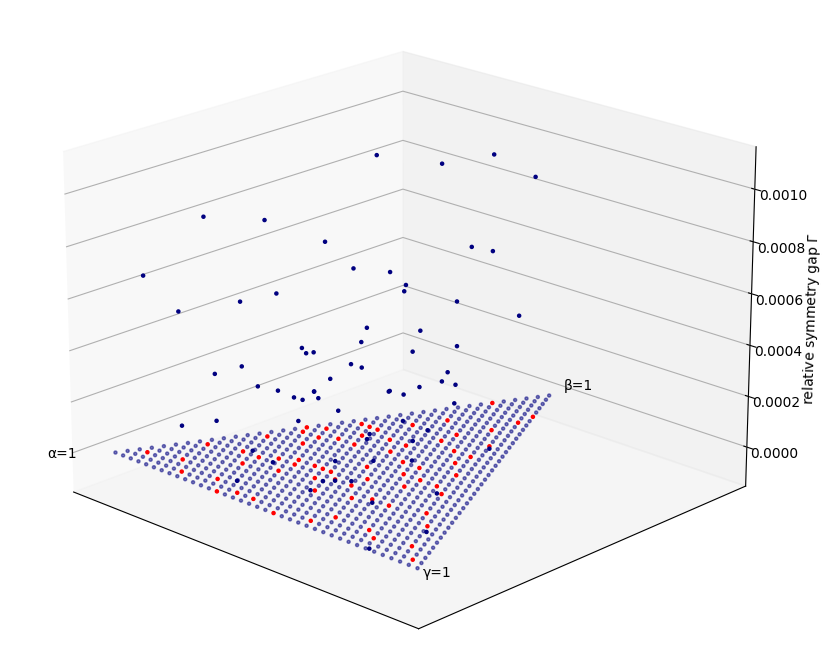}
 \caption[Error of Symmetric Solutions]{Relative gap for $n=8, Y=24000, g=1/3, a=0.8$ and $ \mu \approx 0.878$ corresponding to Fielbaum's results}
 \label{fig:ErrorFielbaum}
 \end{subfigure}
\end{figure}

For this choice of the parameter $\mu,$ asymmetric solutions occur relatively frequently, see, e.g.,  Figure~\ref{fig:typeMapFielbaum} which contains approximately $8.5\%$ of asymmetric optimal line plans. However, the relative symmetry gap is extremely small, the largest gap being less than $0.11\%.$ We conclude that even though symmetry should not be assumed by default, in this particular case, symmetric solutions are close to optimal, and the symmetry assumption of Fielbaum et al.\ is, in this sense, justified.  

\subsection{Computation Time}
\label{subsec:time}

The computation time needed to solve one instance of the Parametric City varies, depending on the input and the model. Table~\ref{tab:computationTimeMILPA} reports the average, minimal, and maximal time needed to compute one of the batches presented above, namely, for the $741$ instances corresponding to Figure~\ref{fig:mu1n8K100}. All computations were performed on a 3.50GHz Intel(R) Xeon(R) CPU E3-1245 v5 machine running Linux. 
All instances could be solved fairly quickly: On average, $ALPP$ was solved in approximately four seconds, but it could take up to $6.25$ minutes. However, this was significantly larger than the time needed to compute the symmetric model.  $ALPP_{\mathcal{S}}$  terminated for all instances in less than half a second. Even if  these computation times are small, the differences are striking: On average, solving $ALPP$ took approximately $145$ times longer than solving $ALPP_{\mathcal{S}}.$ 

\begin{table}[htbp]
	\centering
	\begin{tabular}{|l|  r r |}
		\hline
		& $ALPP$ & $ALPP_{\mathcal{S}} $ \\
		\hline
		Average & 4.022  & 0.028  \\
		Minimum & 0.027  & 0.022  \\
		Maximum & 374.686& 0.039 \\
		\hline
	\end{tabular}
\caption[Computation times of the symmetric problem in comparison to $ALPP$]{Computation times in seconds: symmetric model $ALPP_S$ in comparison to the general model $ALPP$ for $741$ instances of the Parametric City (cf. Figure~\ref{fig:mu1n8K100}) 
}
\label{tab:computationTimeMILPA}
\end{table}

\section{Conclusions}
\label{sec:conclusion}

We investigated the line planning problem in the Parametric City, which combines a versatile generic city planning paradigm with a flexible optimization model allowing the construction of arbitrary lines. Although the Parametric City is inherently rotation symmetric with respect to both shape and demand, it turns out that optimal line plans are not necessarily symmetric. However, the differences between symmetric and asymmetric optimal solutions are rather small in practice, while in theory an approximation algorithm can derived (fixing one of the parameters), building upon the polynomial-time solvability of the line planning problem in Parametric City when line frequencies have to obey rotation symmetry. From a planners' perspective, restricting to symmetric line plans is hence a preeminent starting point to design a transportation network.

\bibliography{references.bib}

\begin{thebibliography}{}

\bibitem[Assad, 1980]{Assad}
Assad, A.~A. (1980).
\newblock Modelling of rail networks: Toward a routing/makeup model.
\newblock {\em Transportation Research Part B: Methodological}, 14(1):101 --
  114.

\bibitem[Badia, 2020]{Badia}
Badia, H. (2020).
\newblock Comparison of bus network structures in face of urban dispersion for
  a ring-radial city.
\newblock {\em Networks and Spatial Economics}, 20:233–271.

\bibitem[Bornd\"{o}rfer et~al., 2007]{Borndorfer:Column}
Bornd\"{o}rfer, R., Gr\"{o}tschel, M., and Pfetsch, M.~E. (2007).
\newblock A column-generation approach to line planning in public transport.
\newblock {\em Transportation Science}, 41(1):123--132.

\bibitem[Bussieck et~al., 1997]{Bussieck}
Bussieck, M., Winter, T., and Zimmermann, U. (1997).
\newblock Discrete optimization in public rail transport.
\newblock {\em Mathematical Programming}, 79:415--444.

\bibitem[Byrne, 1975]{BYRNE197597}
Byrne, B.~F. (1975).
\newblock Public transportation line positions and headways for minimum user
  and system cost in a radial case.
\newblock {\em Transportation Research}, 9(2):97--102.

\bibitem[Ceder and Wilson, 1986]{CEDER1986}
Ceder, A. and Wilson, N.~H. (1986).
\newblock Bus network design.
\newblock {\em Transportation Research Part B: Methodological}, 20(4):331 --
  344.

\bibitem[Daganzo, 2010]{Daganzo}
Daganzo, C.~F. (2010).
\newblock Structure of competitive transit networks.
\newblock {\em Transportation Research Part B: Methodological}, 44(4):434--446.

\bibitem[Fielbaum et~al., 2016a]{FIELBAUM2016298}
Fielbaum, A., Jara-Díaz, S., and Gschwender, A. (2016a).
\newblock Optimal public transport networks for a general urban structure.
\newblock {\em Transportation Research Part B: Methodological}, 94:298 -- 313.

\bibitem[Fielbaum et~al., 2016b]{Fielbaum:ModelIntro}
Fielbaum, A., Jara-Díaz, S., and Gschwender, A. (2016b).
\newblock A parametric description of cities for the normative analysis of
  transport systems.
\newblock {\em Networks and Spatial Economics}, 17.

\bibitem[Fielbaum et~al., 2018]{FielbaumTransitLineStructures}
Fielbaum, A., Jara-Díaz, S., and Gschwender, A. (2018).
\newblock Transit line structures in a general parametric city: The role of
  heuristics.
\newblock {\em Transportation Science}, 52.

\bibitem[Fielbaum~Schnitzler, 2019]{Fielbaum:Thesis}
Fielbaum~Schnitzler, A.~S. (2019).
\newblock {\em Effects of the Introduction of Spatial and Temporal Complexity
  on the Optimal Design, Economies of Scale and Pricing of Public Transport}.
\newblock PhD thesis, Universidad de Chile.

\bibitem[Gallian, 2000]{helm}
Gallian, J. (2000).
\newblock A dynamic survey of graph labeling.
\newblock {\em Electron J Combin DS6}, 19.

\bibitem[{Gurobi Optimization, LLC}, 2020]{gurobi}
{Gurobi Optimization, LLC} (2020).
\newblock Gurobi optimizer reference manual.

\bibitem[Holroyd, 1967]{Holroyd}
Holroyd, E. (1967).
\newblock The optimum bus service: A theoretical model for a large uniform
  urban area.
\newblock {\em Proceedings of the 3rd International Symposium on the Theory of
  Road Traffic Flow}, pages 308--328.

\bibitem[Jansson, 1980]{Jansson1980}
Jansson, J.~O. (1980).
\newblock A simple bus line model for optimisation of service frequency and bus
  size.
\newblock {\em Journal of Transport Economics and Policy}, 14(1):53--80.

\bibitem[Jara-Díaz and Gschwender, 2009]{JaraDiazGschwender}
Jara-Díaz, S. and Gschwender, A. (2009).
\newblock The effect of financial constraints on the optimal design of public
  transport services.
\newblock {\em Transportation}, 36:65--75.

\bibitem[Karbstein, 2013]{Karbstein:Thesis}
Karbstein, M. (2013).
\newblock {\em Line Planning and Connectivity}.
\newblock PhD thesis, Technische Universität Berlin.

\bibitem[Kepaptsoglou and Karlaftis, 2009]{KepaptsoglouKarlaftis}
Kepaptsoglou, K. and Karlaftis, M. (2009).
\newblock Transit route network design problem: Review.
\newblock {\em Journal of Transportation Engineering}, 135(8):491--505.

\bibitem[Lenstra, 1983]{Lenstra}
Lenstra, H.~W. (1983).
\newblock Integer programming with a fixed number of variables.
\newblock {\em Mathematics of Operations Research}, 8(4):538--548.

\bibitem[Lopez, 2014]{Lopez}
Lopez, F. (2014).
\newblock Integrating network design and frequency setting in public
  transportation networks: A survey.
\newblock {\em SORT (Statistics and Operations Research Transactions)},
  38:181--214.

\bibitem[Schmidt, 2014]{Schmidt}
Schmidt, M.~E. (2014).
\newblock {\em Integrating routing decisions in public transportation
  problems}.
\newblock Springer.

\bibitem[Schöbel, 2011]{Schoebel:overview}
Schöbel, A. (2011).
\newblock Line planning in public transportation: Models and methods.
\newblock {\em OR Spektrum}, 34:1--20.

\end{thebibliography}

\appendix
\section{Appendix}
\label{sec:appendix}

\begin{proof}[Proof of Proposition~\ref{prop:explicitSolutionToSP}]

The cost values correspond to those of the uncapacitated minimum-cost flow problem corresponding to Figure~\ref{fig:shortestpathvalues}. For better readability, let us summarize arc costs of the same type:
\begin{align*}
\bar{c}_{ps} &:= \bar{c}_{(P_i,SC_i)}   = (2\mu/K +(1-\mu)) T \, g,\\
\bar{c}_{sp} &:= \bar{c}_{(SC_i,P_i)}   = (1-\mu) T \, g,\\
\bar{c}_{ss} &:= \bar{c}_{(SC_i,SC_{i\pm 1})}   = (\mu/K + (1-\mu)) T \, r_n,\\
\bar{c}_{sc} &:= \bar{c}_{(SC_i,CD)}  = (2\mu/K + (1-\mu)) T,\\
\bar{c}_{cs} &:=\bar{c}_{(CD, SC_i)}   = (1-\mu) T
\end{align*}
for all  $i\in \ZnZ$.

It is obvious that for an OD pair  $(SC_i, SC_{i \pm j})$ the shortest path is either via the central business district $(SC_i, CD, SC_j)$, which means that it contributes $2 T (\mu/K+(1-\mu))$ to the objective value, or along the subcenter ring -- $(SC_i, SC_{i\pm 1}, \dots, SC_{i\pm j})$ -- with cost $j T r_n (\mu/K +(1-\mu)).$
Thus, for $j \in \{1,\dots, k_n\}$ with $k_n:=\left\lfloor \frac{2}{r_n} \right\rfloor$ the shortest path from $SC_i$ to $SC_{i+j}$ leads over neighboring subcenters, while the rest pass through $CD.$ The same holds for the OD-pairs $(P_i, SC_{i\pm j}).$
An overview of the shortest paths is depicted in Figure~\ref{fig:shortestpaths}, with corresponding costs $\bar{c}_{s \to t}$  in Table~\ref{tab:shortestPath}.

\begin{table}[ht]
    \begin{center}
        \begin{tabular}{|l|l| l|l|}
            \hline
            $s$ & $t$ & shortest path & cost $\bar{c}_{s\to t} $\\
            \hline
            $P_i$ & $SC_i$ & $(P_i,SC_i)$ & $\bar{c}_{ps}$ \\
            $P_i$ & $SC_{i\pm j}, j\in \{1,\dots, k_n\}$ & $(P_i, SC_i, SC_{i\pm 1}, \dots, SC_{i\pm j})$ & $\bar{c}_{ps} + j\, \bar{c}_{ss}$\\
            $P_i$ &  $SC_{i\pm j}, j\in \{k_n,\dots, \lfloor \frac{n}{2} \rfloor\}$ & $(P_i, SC_i, CD, SC_{i\pm j})$ & $\bar{c}_{ps} +  \bar{c}_{cs} + \bar{c}_{ss}$\\
            $P_i$ & $CD$ & $(P_i,SC_i, CD)$ & $\bar{c}_{ps}+\bar{c}_{cs}$ \\
            $SC_i$ & $SC_{i\pm j}, j\in \{1,\dots, k_n\}$ & $(SC_i, SC_{i\pm 1}, \dots, SC_{i\pm j})$ &  $j\, \bar{c}_{ss}$\\
            $SC_i$ &  $SC_{i\pm j}, j\in \{k_n,\dots, \lfloor \frac{n}{2} \rfloor\}$ & $(SC_i, CD, SC_{i\pm j})$ & $\bar{c}_{sc} + \bar{c}_{cs}$\\
            $SC_i$ & $CD$ & $(SC_i, CD)$ & $\bar{c}_{cs}$ \\
            \hline
        \end{tabular}
    \end{center}
    \caption{Shortest paths for each origin-destination pair $(s,t)\in D$}
    \label{tab:shortestPath}
\end{table}

First, consider only the cost of passengers traveling from $SC_0$ to any of the other subcenters:
\begin{align*}
\sum_{j=1}^n \bar{c}_{SC_0\to SC_j} \, d_{SC_0,SC_j}   &= 2 \sum_{j = 1}^{k_n} j \,  \bar{c}_{ss} \, d_{SC_0, SC_j} + \sum_{j = k_n+1}^{n-(k_n+1)} (\bar{c}_{sc}+ \bar{c}_{cs}) \, d_{SC_0, SC_j}  \\
&= 2\frac{k_n(k_n+1)}{2} \bar{c}_{ss} \, d_{SC_0, SC_j} + (n-2k_n+1)  (\bar{c}_{sc}+ \bar{c}_{cs}) \, d_{SC_0, SC_1} \\
& =\left( k_n(k_n+1) \bar{c}_{ss}+ (n-2k_n+1)  (\bar{c}_{sc}+ \bar{c}_{cs})\right) \, d_{SC_0, SC_1}.
\end{align*}

The shortest path costs for passengers starting at one periphery to all foreign subcenters can be calculated similarly. Since each path has to take arc $(P_0, SC_0),$ the term $\bar{c}_{ps}$ needs to be added for each of the $n-1$ paths. Thus we get
\begin{align*}
&\sum_{j=1}^n \bar{c}_{P_0\to SC_j} \, d_{P_0,SC_j} =    \left( k_n(k_n+1) \bar{c}_{ss}+ (n-2k_n+1)  (\bar{c}_{sc}+ \bar{c}_{cs}) + \bar{c}_{ps}(n-1) \right) \, d_{P_0,SC_1}.
\end{align*}
Passengers only start in either a periphery or a subcenter, none originate at the center. Thus one can calculate the explicit costs of a single zone simply by adding these two terms and inserting the values for the demand $d:$
\begin{align*}
\sum_{v \in V}& d_{P_0,v}  \bar{c}_{P_0\to v} + \sum_{v \in V} d_{S_0,v} \bar{c}_{SC_0 \to v}\\
& = T \, \frac{Y}{n} \left(\Big( k_n(k_n+1)r_n+ 2 (n-2k_n+1) \Big) \left(\frac{\mu}{K} + (1-\mu)\right)  \, \left(\frac{a\gamma+(1-a) \tilde{\gamma}}{(n-1)}\right) \right.\\
&\left. \hspace{.5cm} + \left(\frac{2\mu}{K} + (1-\mu)\right)   \left(a\alpha+(1-a)\tilde{\alpha}\right) + \left(\frac{2\mu}{K} + (1-\mu)\right)g \, a \right)\\
&= T\, \frac{Y}{n}\lambda(\alpha, \gamma).
\end{align*}
This implies that the cost for zone $0$ depends linearly on $Y/n$ -- the total amount of people from a single zone.
As the costs are the same for each zone, the explicit optimal value of the shortest path problem must be $n$ times the cost of a single zone, where each passenger takes the shortest path:
\begin{align*}
cost_{UMCFP}(y) &= \sum_{p \in P} \bar{c}_p y_p 
=  T \, Y\lambda(\alpha, \gamma). \qedhere
\end{align*}
\end{proof}

    \begin{proof}[Proof of Lemma \ref{lem:lowerBound}]
    Let $RLPP$ denote the described relaxation of $ALPP,$ i.e.,
   
        \begin{mini!}|s|[2]<b>
        {F,y}{ \sum_{a \in A} \tau_a F_a + \sum_{p \in P} {\tau}_p \textcolor{black}{y_p} = cost_{RLPP}(F,y)}
        {}
        {\textcolor{black} {(RLPP)\quad} \nonumber}
        \addConstraint{\sum_{p \in P_{st}} \textcolor{black}{y_p}}{= d_{s,t} \quad \label{con:demand_sp2}} {\forall (s,t) \in D}
            \addConstraint{\sum_{p \in P_a}{y_p} - K F_a }{\geq 0 \label{con:paths_lb1}\quad}{\forall a \in A }
        \addConstraint{\textcolor{black}{y_p} }{\geq 0 \label{con:paths_lb2}}{\forall p \in P}
        \addConstraint{\textcolor{black}{F_a} }{\geq 0 \label{con:paths_lb3}}{\forall a \in A}
    \end{mini!}

    Suppose we have a solution $(F,y)$ to $RLPP.$ Clearly, $y$ is feasible for $UMCFP$ as well. What is left to show is that $\OptVal(UMCFP) \leq \OptVal(RLPP).$
    We know that $F_{(P_i, SC_i)} = F_{(SC_i, P_i)};$ with Condition~\ref{con:arcCapacity_A} we get
    \begin{equation}
    \sum_{p\in P_{(P_i, SC_i)}}\frac{y_p}{K} \leq F_{(P_i, SC_i)} = F_{(SC_i, P_i)}
    \label{eq:Fps}
    \end{equation}
    for all $i\in \ZnZ.$
    Similarly,  $\sum_{i=0}^{n-1} F_{(SC_i, CD)} = \sum_{i=0}^{n-1} F_{(CD,SC_i)}$ has to hold. This results in
    \begin{equation}
    \max\left\{\sum_{i=0}^{n-1} \sum_{p\in P_{(SC_i, CD)}} \frac{y_p}{K}, \sum_{i=0}^{n-1} \sum_{p\in P_{(CD, SC_i)}} \frac{y_p}{K} \right\}\leq \sum_{i=0}^{n-1} F_{(SC_i, CD)} =  \sum_{i=0}^{n-1} F_{(CD, SC_i)}
    \label{eq:max}
    \end{equation}
    in order to fulfill the capacity constraints.
    
   
    For the subcenter arc-frequencies, we obtain analogously:
    \begin{equation}
    \sum_{p\in P_{(SC_{i}, SC_{i\pm 1})}}\frac{y_p}{K} \leq F_{(SC_i, SC_{i \pm 1})}. \label{eq:Fss}   
    \end{equation}
    Note, however, that both directions need to be considered separately.
   
    By applying inequalities~\eqref{eq:Fps}, \eqref{eq:max} and \eqref{eq:Fss} to each of the corresponding arc-frequencies, as well as inserting the values of the arc-costs, we can estimate $\sum_{a \in A} \tau_a F_a$ as
    \begin{align*}
    \sum_{a \in A} \tau_a  F_a
    &\geq
    2\,T\,g\sum_{i = 0}^{n-1} \sum_{p\in P_{(P_i, SC_i)}} \frac{y_p}{K} 
    + 2 \, T\sum_{i = 0}^{n-1}  \sum_{p\in P_{(SC_i, CD)}} \frac{y_p}{K} \\
    &+ T \, r_n \sum_{i = 0}^{n-1} \sum_{p \in P_{(SC_{(i+1)},SC_i)}} \frac{y_p}{K}
    + T \, r_n \sum_{i = 0}^{n-1} \sum_{p \in P_{(SC_{(i-1)},SC_i)}} \frac{y_p}{K} .
    \end{align*}
   
    By plugging this into the cost function of $RLPP$ and some simple rearrangement, we have
    {$$cost_{RLPP}(F,y) = \mu \sum_{a \in A} \tau_a F_a + (1-\mu) \sum_{p \in P} \tau_p y_p = \sum_{a \in A} \sum_{p\in P_a} \bar{c}_a y_p = \sum_{p\in P} \bar{c}_p y_p = cost_{UMCFP}(y),$$ \nopagebreak
        from which the claim follows.}
\end{proof}

    \begin{proof}[Proof of Corollary \ref{cor:boundSP}]
    Consider the formula for the optimal value of the uncapacitated mini\-mum-cost flow problem $\OptVal(UMCFP) = TY\lambda(\alpha,\gamma) $ from Proposition~\ref{prop:explicitSolutionToSP}.  We will analyze the term
    $$\frac{k_n(k_n+1)r_n +2(n-2k_n+1)}{(n-1)}$$ and
    show that $\frac{k_n(k_n+1)r_n}{(n-1)} \in [\frac{2}{\pi}, 2]$ and
    $\frac{n-2k_n+1}{n-1} \in [1-\frac{2}{\pi}, 1]$ for all $n\geq4:$
    As $r_n T$ is the distance of two subcenters and $k_n =\left \lfloor \frac{2}{r_n}\right\rfloor,$ i.e., the largest number such that the length of $k_n$ subcenter-arcs is shorter than two central arcs, we have $k_n r_n \leq 2$ and $(k_n+1) r_n >2.$ Further, we have the bound
    $r_n n \leq 2 \pi,$ because $r_n$ is the side length of an $n$-gon with vertices on the unit circle, thus the perimeter is less than $2 \pi.$  By a similar argument we can inscribe a circle of radius $\cos\left(\frac{\pi}{n}\right)$ within that regular $n$-gon such that the sides of the polygon touch the circle tangentially.
    The perimeter of the $n$-gon must therefore be larger than the circumference of the inscribed circle. As the radius of the inscribed circle increases with $n,$ its minimum is at $n=4,$ thus we have $$n r_n >  2 \pi \cos\left(\frac{\pi}{n}\right)\geq  2 \pi \cos\left(\frac{\pi}{4}\right) \geq \sqrt{2} \pi.$$
    Further, we use the fact that $n-2k_n +1 \leq n-1$ and $k_n +1 \leq n-1.$
    Therefore, for $n\geq 4$ we have the three estimates
    \begin{align*}
    \frac{k_n(k_n+1)r_n}{n-1} &\leq \frac{2(k_n+1)}{n-1} \leq{2},\\
    \frac{k_n(k_n+1)r_n}{n-1} &\geq \frac{2k_n}{n-1} \geq \frac{2k_n}{2\pi/r_{n-1}} \geq \frac{r_{n-1} k_n}{\pi} \geq \frac{r_{n-1} k_{n-1}}{\pi} \geq \frac{2}{\pi},\\
    \frac{n-2k_n+1}{n-1} &\leq 1.
    \end{align*}
    For the lower bound on $\frac{n-2k_n+1}{n-1}$ we need to examine the term in more detail. We can derive the following inequality by making use of the property $n r_n >  2 \pi \cos\left(\frac{\pi}{n}\right):$
    \begin{align*}
    \frac{n-2k_n+1}{n-1} &\geq \frac{n-2k_n+1}{n}
    = \frac{n - 2 \left\lfloor \frac{2}{r_n} \right\rfloor +1}{n} \geq 1 - \frac{4}{r_n \, n} + \frac{1}{n} \geq 1-\frac{2}{\pi \cos(\frac{\pi}{n})} +\frac{1}{n}.
    \end{align*}
    For the last term, define the help function $h:[4,\infty[ \rightarrow \mathbb{R}$ with $h(x) =  1-\frac{2}{\pi \cos(\frac{\pi}{x})} +\frac{1}{x}.$ It is possible to show that $h$ is monotonically decreasing on the interval $[8, \infty[:$ The function is differentiable with the first derivative
    \begin{align*}
    \dfrac{d}{dx} h(x) &= \dfrac{1}{x^2} \left( \dfrac{2\sin\left(\frac{{\pi}}{x}\right)}{\cos^2\left(\frac{{\pi}}{x}\right)}-1\right).
    \end{align*}
    On the interval $[8, \infty[,$ both $\sin\left(\frac{{\pi}}{x}\right)$ and $\cos\left(\frac{\pi}{x}\right)$ are positive; the former is monotonically decreasing, while the latter (and thus $\cos^2(\frac{\pi}{x})$ as well) is  monotonically increasing. Hence, the following holds true
    $$\dfrac{2\sin\left(\frac{{\pi}}{x}\right)}{\cos^2\left(\frac{{\pi}}{x}\right)}  < \dfrac{2\sin\left(\frac{{\pi}}{8}\right)}{\cos^2\left(\frac{{\pi}}{8}\right)} = \dfrac{2 \sqrt{2 - \sqrt{2}}}{2+\sqrt{2}} \approx 0.448 < 1$$
    for all $x\geq 8.$ This implies that on the interval $[8, \infty[,$   $h$ is monotonically decreasing since its derivative is negative.  It is clear that the limit exists with $$\lim_{x \to \infty} h(x) = 1-\dfrac{2}{\pi}.$$
    %
    %
    %
    Thus we have $\frac{n-2k_n+1}{n-1}  \geq 1-\frac{2}{\pi}$ for all $ n \in \mathbb{N}_{\geq8}.$
    For $n \in \{4,5,6,7\}$ the bound can be verified easily by plugging in the values $k_4=k_5 = 1$ and $k_6 = k_7 = 2.$
    Together, the following inequality holds for all $n \in\mathbb{N}_{\geq 4}:$
    $$2-\frac{2}{\pi} \leq \frac{k_n(k_n+1)r_n + 2(n-2k_n+1)}{n-1} \leq 4 .$$
    Finally, we can estimate $\lambda(\alpha, \gamma)$ by making use of the properties $\alpha+\beta+\gamma=1$ and $\tilde{\alpha} + \tilde{\gamma} = 1:$
    \begin{align*}
    \lambda(\alpha,\gamma) &\leq  4 \left(\frac{\mu}{K} + (1-\mu)\right)  \, \left({a\gamma+(1-a) \tilde{\gamma}}\right) \\
    &\hspace{.5cm} + \left(\frac{2\mu}{K} + (1-\mu)\right)   \left(a\alpha+(1-a)\tilde{\alpha}\right) + \left(\frac{2\mu}{K} + (1-\mu)\right)ga \\
    &\leq 4 \left(\frac{2\mu}{K} + (1-\mu)\right) (1+ga)\\
    \end{align*}and
    \begin{align*}
    \lambda(\alpha,\gamma)&\geq \left(2-\frac{2}{\pi} \right) \left(\frac{\mu}{K} + (1-\mu)\right)  \, \left({a\gamma+(1-a) \tilde{\gamma}}\right) \\
    &\hspace{.5cm} + \left(\frac{2\mu}{K} + (1-\mu)\right)   \left(a\alpha+(1-a)\tilde{\alpha}\right) + \left(\frac{2\mu}{K} + (1-\mu)\right)g a \\
    &\geq  \left( \left(2-\frac{2}{\pi} \right)\frac{\mu}{K} + (1-\mu)\right) (1+ga-a),\\
    \end{align*}
    from which the claim follows.
\end{proof}   

\newpage
\section{List of Symbols}
\label{sec:listOfSym}
 \addcontentsline{toc}{chapter}{List of Symbols}
\markboth{List of Symbols}{List of Symbols}
\begin{longtabu} to \linewidth {p{0.14\linewidth} p{0.7\linewidth} p{0.06\linewidth}}
	\textbf{Symbol} &  \textbf{Description} & \textbf{Page}  \\
	$A$ & set of arcs in the Parametric City & \pageref{ass:A} \\
	$a$ & fraction of people starting in the peripheries & \pageref{ass:a}\\
		$ALPP$ & arc-based line planning problem & \pageref{MILP:A} \\
	$CD$ & central business district  & \pageref{ass:CD} \\
	$\tau_a$, $\tau_l$, $\tau_p$ & cost of/time to travel along arc $a$, line $l$, or path $p$ & \pageref{ass:tau},\pageref{ass:tauL},\pageref{ass:tauP}\\
	$\bar{c}_a, \bar{c}_p$ & cost parameters of $UMCFP$ & \pageref{UMCFP}\\
	$d_{s,t}$ & demand from node $s$ to $t$ & \pageref{ass:d}\\
	$F$ & frequency variables of arcs of $ALPP$  & \pageref{ass:F}\\
	$F_{C}$ & frequency variable of $ALPP_{\mathcal{S}}$ corresp. to central arcs & \pageref{ass:FCSP} \\
	$F_{P}$ & frequency variable of $ALPP_{\mathcal{S}}$ corresp. to peripheral arcs & \pageref{ass:FCSP}\\
	$F_{S+}$ & frequency variable of $ALPP_{\mathcal{S}}$ corresp. to arcs $(SC_j, SC_{j+1})$& \pageref{ass:FCSP} \\
	$F_{S-}$ & frequency variable of $ALPP_{\mathcal{S}}$ corresp. to arcs $(SC_j, SC_{j-1})$ & \pageref{ass:FCSP}\\
	$f$ & frequency variables of lines of $LPP$ & \pageref{ass:f}\\
	$\mathcal{G} = (V,A)$ & graph of the Parametric City &  \pageref{ass:G} \\
	$g$ & factor for distance between $SC_j$ and $P_j$ & \pageref{ass:g} \\
	$K$ & vehicle capacity for number of passengers &  \pageref{ass:K} \\
	$k_n$ & largest value s.t. the shortest $SC_0$-$SC_{k_n}$-path is along subcenters& \pageref{prop:explicitSolutionToSP} \\
	$L$ & set of lines corresponding to directed circuits & \pageref{ass:L} \\
	$L_{a} $ & set of lines using arc $a$ &  \pageref{ass:La} \\
	$LPP$ & line-based problem & \pageref{MILP} \\
	$n$ & number of zones in the Parametric City & \pageref{ass:n} \\
	$\OptVal(P)$ & optimal value of problem $P$ & \pageref{ass:OptVal} \\
	$P$ & set of (passenger) paths corresponding to directed simple paths & \pageref{ass:P} \\
	$P_a$ & set of paths using arc $a$ & \pageref{ass:Pa} \\
	$P_{s\to t}$ & set of paths with origin $s$ and destination $t$ & \pageref{ass:Pst} \\
	$P_j$ & periphery of zone $j$ & \pageref{ass:Pi} \\
	$r_n$ & factor for distance between two subcenters & \pageref{ass:r} \\
	$SC_j$ & subcenter of zone $j$ & \pageref{ass:S}\\
	$T$ & distance between $CD$ and a subcenter & \pageref{ass:T} \\
	$UMCFP$ & uncapacitated minimum-cost flow problem as relaxation of $ALPP$ & \pageref{UMCFP} \\
	$V$ & set of nodes in the Parametric City & \pageref{ass:A} \\
	$Y$ & total patronage & \pageref{ass:Y}\\
	$y$ & passenger variables of $LPP, ALPP, ALPP_{\mathcal{S}}, UMCFP$ & \pageref{ass:y}\\
	$\alpha$ & fraction of people traveling from $P_j$ to $CD$ & \pageref{ass:alphaBetaGamma}\\
	$\tilde{\alpha}$ & fraction of people traveling from $SC_j$ to $CD$ & \pageref{ass:tildeAlpha}\\
	$\beta$ & fraction of people traveling from $P_j$ to $SC_j$& \pageref{ass:alphaBetaGamma}\\
	$\gamma$ & fraction of people traveling from $P_j$ to other subcenters& \pageref{ass:alphaBetaGamma}\\
	$\tilde{\gamma}$ & fraction of people traveling from $SC_j$ to other subcenters& \pageref{ass:tildeGamma}\\
	$\delta^{+}(v), \delta^-(v)$ & set of outgoing, incoming arcs at node $v$ & \pageref{ass:delta} \\
	$\Lambda$ & frequency capacity for arcs & \pageref{ass:Lambda}\\
	$\lambda(\alpha,\gamma)$ & scaled optimal value of $UMCFP$ &\pageref{prop:explicitSolutionToSP}\\
	$\mu$ & factor to weigh operator against user costs& \pageref{ass:mu}\\
	&  $\mu = 0$  means no operator costs, $\mu = 1$ no user costs&\\
	$\rho_z(x)$ & rotation of vertex tuple $x$ by $z \, 2 \pi/n $ & \pageref{ass:rotation}\\
	$a\in p, a\in l$ & path $p$, line $l$ uses arc $a$ & \pageref{ass:in} \\
\end{longtabu}


\end{document}